\documentclass[12pt]{article}
\usepackage{tikz}
\usetikzlibrary{arrows}
\usepackage[framemethod=tikz]{mdframed}
\usepackage{wrapfig}
\usepackage{amsmath,amssymb}
\usepackage{type1cm}
\usepackage{amsthm}
\usepackage{mathrsfs}
\usepackage{mathtools}
\allowdisplaybreaks[1]
\usepackage{enumerate}
\usepackage[all]{xy} 
\usepackage[vcentermath,enableskew]{youngtab}
\usepackage{ytableau}
\usepackage{diagbox}
\AtBeginDocument{%
   \def\MR#1{}
}

\DeclareMathOperator{\GL}{GL}

\DeclareMathOperator{\GU}{GU}

\DeclareMathOperator{\diag}{diag}

\DeclareMathOperator{\Ad}{Ad}
\DeclareMathOperator{\Adm}{Adm}

\DeclareMathOperator{\Irr}{Irr}

\DeclareMathOperator{\pfn}{pfn}

\DeclareMathOperator{\op}{op}

\DeclareMathOperator{\supp}{supp}

\DeclareMathOperator{\LP}{LP}

\DeclareMathOperator{\DL}{DL}

\newcommand\F{\mathbb{F}}

\newcommand\aFq{\overline{\mathbb F}_q}

\newcommand\cO{\mathcal O}

\newcommand\Gm{\mathbb G_m}

\newcommand\Q{\mathbb Q}

\newcommand\A{\mathbb A}

\newcommand\G{\mathbb G}
\newcommand\J{\mathbb J}

\newcommand\Z{\mathbb Z}
\newcommand\tW{\tilde W}

\newcommand\SW{{^S\tilde W}}
\newcommand\SAdm{{^S\mathrm{Adm}}}

\newcommand\tS{\tilde S}

\newcommand\vp{\varpi}
\newcommand\Y{X_*(T)}

\newcommand\la{\langle}
\newcommand\ra{\rangle}

\theoremstyle{definition}
\newtheorem{theo}{Theorem}[section]
\newtheorem{prop}[theo]{Proposition}

\newtheorem{lemm}[theo]{Lemma}
\newtheorem{coro}[theo]{Corollary}
\newtheorem{exam}[theo]{Example}
\newtheorem{rema}[theo]{Remark}

\newtheorem{thm}{Theorem}[section]

\begin{document}
\title{On the Supersingular Locus of the $\GU(2,n-2)$ Shimura Variety}
\author{Ryosuke Shimada}
\date{}
\maketitle

\begin{abstract}
We study the supersingular locus of a reduction at an inert prime of the Shimura variety attached to $\GU(2,n-2)$.
More concretely, we decompose the supersingular locus into a disjoint union of iterated fibrations over (classical) Deligne-Lusztig varieties after taking perfection.
As an immediate application, when $p\geq3$, we prove an analogue of Oort's conjecture on the supersingular locus of the Siegel modular variety.
\end{abstract}

\section{Introduction}
\label{introduction}
Shimura varieties have been used, with great success, towards applications in number theory.
There are many such applications based on the study of integral models and their reductions.
It is known that in some cases, the supersingular (or basic) locus of the reduction of a Shimura variety admits a simple description.
For example, Vollaard-Wedhorn \cite{VW11} described the supersingular locus of the Shimura variety of $\GU(1,n-1)$ at an inert prime as a union of (classical) Deligne-Lusztig varieties.
Also in the $\GU(2,2)$-case, Howard-Pappas \cite{HP14} proved the existence of a similar description.
After \cite{VW11} and \cite{HP14}, G\"{o}rtz, He and Nie classified the cases where the supersingular locus is naturally a union of Deligne-Lusztig varieties, called the {\it fully Hodge-Newton decomposable} cases (cf.\ \cite{GH15}, \cite{GHN19}, \cite{GHN24}).
As a result, the Shimura variety of $\GU(2,n-2)$ is not fully Hodge-Newton decomposable if $n\geq 5$.
The studies by G\"{o}rtz, He and Nie are based on the fact that the study of the perfection of the supersingular locus can be reduced to a study of an affine Deligne-Lusztig variety via the Rapoport-Zink uniformization.
Such simple descriptions have been applied towards the Kudla-Rapoport program \cite{KR11}, Zhang's Arithmetic Fundamental Lemma \cite{Zhang12} and the Tate conjecture for certain Shimura varieties \cite{TX19}, \cite{HTX17}.

Recently, new simple descriptions have been discovered in some cases which are not fully Hodge-Newton decomposable (cf.\ \cite{Trentin23}, \cite{Shimada5}, \cite{ST24}).
In the $\GU(2,n-2)$-case, Fox-Imai \cite{FI21} (see \cite{FHI23} for a result before perfection) studied irreducible components of the supersingular locus using the Chen-Zhu conjecture, which is a theorem on irreducible components of affine Deligne-Lusztig varieties.
On the other hand, any stratification (i.e., a decomposition into disjoint locally closed subvarieties) as in the fully Hodge-Newton decomposable cases was not known in this case.
Hence the objective of this paper is to find an explicit stratification of the affine Deligne-Lusztig variety associated to the Shimura variety of $\GU(2,n-2)$ in terms of Deligne-Lusztig varieties (see \cite[\S 11]{FI21} for the Rapoport-Zink uniformization in this case).

Let $F$ be a non-archimedean local field with finite residue field $\F_q$ of prime characteristic $p$, and let $L$ be the completion of the maximal unramified extension of $F$.
We write $\cO$ for the valuation ring of $L$.
To simplify the exposition, we assume that $F$ has mixed characteristic in the introduction.
Let $G$ be the unramified general unitary group of degree $n$ over $F$.
Let $\mu$ be the cocharacter of $G$ corresponding to $z\rightarrow (\diag(1,\ldots,1,z^{-1},z^{-1}),z^{-1})$ under an isomorphism $G_L\cong \GL_n\times \G_m$.
Let $X_\mu(b)$ denote the affine Deligne-Lusztig variety attached to $\mu$ and $b=(\diag(1,\ldots,1),\vp^{-1})$, where $\vp$ denotes a uniformizer of $F$.
Including $X_\mu(b)$, all varieties in the introduction are perfect schemes.
For example, Deligne-Lusztig varieties actually mean the perfection of them.

In the fully Hodge-Newton decomposable cases, the decomposition into Deligne-Lusztig varieties is a refinement of the {\it Ekedahl-Oort stratification} of affine Deligne-Lusztig varieties (see \S\ref{ADLV}).
This stratification itself exists in general even outside the fully Hodge-Newton decomposable cases.
It is the local analogue of the stratification defined in the global context of Shimura varieties in \cite{HR17}.
An Ekedahl-Oort stratum in $X_\mu(b)$ actually corresponds to the intersection of a global Ekedahl-Oort stratum with the supersingular locus (cf.\ \cite[\S2.5]{GHR20}).

Set $\vp^\mu=(\diag(1,\ldots,1,\vp^{-1},\vp^{-1}),\vp^{-1})$.
Then the Ekedahl-Oort stratification of $X_\mu(b)$ (and the corresponding global one) is parametrized by $$\SAdm(\mu)=\{w_{k,l}\coloneqq\vp^\mu (n-1\ \cdots\ k+1\ k)(n\ \cdots\ l+1\ l)\mid 1\le k<l\le n\},$$
which is a subset of the Iwahori-Weyl group.
Let us denote by $\pi(X_{w_{k,l}}(b))$ the Ekedahl-Oort stratum corresponding to $w_{k,l}$.
See \S\ref{ADLV} for the precise definition.
Let $\SAdm(\mu)_{\neq \emptyset}\coloneqq\{w_{k,l}\in \SAdm(\mu)\mid \pi(X_{w_{k,l}}(b))\neq \emptyset\}$.
Let $\SAdm(\mu)_{\DL}\subseteq \SAdm(\mu)_{\neq \emptyset}$ denote a certain subset consisting of all $w_{k,l}$ such that $\pi(X_{w_{k,l}}(b))$ is naturally a union of Deligne-Lusztig varieties (see Proposition \ref{spherical} and \S\ref{setting}).
The equality holds if and only if $n\le 4$.
It is also known by G\"{o}rtz-He-Nie that the global Ekedahl-Oort stratum corresponding to $w_{k,l}\in\SAdm(\mu)$ is contained in the supersingular locus if and only if $w_{k,l}\in\SAdm(\mu)_{\DL}$ (cf.  \cite[Proposition 5.6]{GHN19}, \cite[Proposition 4.2]{Wang21}).
Let $\SAdm(\mu)_{\neq\DL}$ denote $\SAdm(\mu)_{\neq \emptyset}\setminus \SAdm(\mu)_{\DL}$.
This set parametrizes the global Ekedahl-Oort strata which intersect but are not contained in the supersingular locus (cf.\ \cite[Lemma 7.6]{GHN19}).
The following theorem gives a complete description of $\SAdm(\mu)_{\neq \emptyset}$.
This was only known in fully Hodge-Newton decomposable cases $n\le 4$ and the case $n=5$ studied in \cite[Theorem 6.1]{ABFGGN25}.
See Example \ref{13 14} for the cases where $n=13$ and $14$.

\begin{thm}[Theorem \ref{non-empty}]
\label{non-empty thm}
We have
\begin{align*}
\SAdm(\mu)_{\DL}=\{w_{k,l}\mid \text{$k=1$ or $l\le \tfrac{n+2}{2}$}\}.
\end{align*}
Moreover, $w_{k,l}\in\SAdm(\mu)_{\neq\DL}$ if and only if $3\le k< \frac{n+2}{2}<l\le n-1$ and one of the following conditions is satisfied:
\begin{enumerate}[(i)]
\item $k$ is odd and $k+l\le n+2$.
\item $l\equiv n-1$ (mod $2$) and $k+l\geq n+3$.
\end{enumerate}
\end{thm}
For $w_{k,l}\in \SAdm(\mu)_{\neq \DL}$, let $w'_{k,l}$ denote $w_{k-2,l}$ (resp.\ $w_{k,l-2}$, resp.\ $w_{k-1,l-1}$) if $k+l\le n+2$ (resp.\ $k+l\geq n+4$, resp.\ $k+l=n+3$).
Then $w'_{k,l}\in \SAdm(\mu)_{\neq \emptyset}$ by Theorem \ref{non-empty thm}.
The following theorem relates $\pi(X_{w_{k,l}}(b))$ to $\pi(X_{w'_{k,l}}(b))$.
\begin{thm}[Theorem \ref{main theo}]
\label{DL method thm}
Let $w_{k,l}\in \SAdm(\mu)_{\neq \DL}$.
Then $\pi(X_{w_{k,l}}(b))$ is a Zariski-locally trivial $\A^1$-bundle over $\pi(X_{w'_{k,l}}(b))$.
In particular, if $k+l\le n+2$ (resp.\ $k+l\geq n+3$), then $\pi(X_{w_{k,l}}(b))$ is an iterated fibration of rank $\frac{k-1}{2}$ (resp.\ $k+\frac{l-n-3}{2}$) over $\pi(X_{w_{1,l}}(b))$ (resp.\ $\pi(X_{w_{1,n-k+2}}(b))$), whose fibers are all $\A^1$.
\end{thm}

An iterated fibration is the composite of some Zariski-locally trivial $\A^1$-bundles (cf.\ \S\ref{DL method}).
The variety $X_\mu(b)$ and each Ekedahl-Oort stratum admit an action of $G(F)$.
All of the fibrations appearing in Theorem \ref{DL method thm} are $G(F)$-equivariant.

Combining Theorem \ref{DL method thm} with the description of $\pi(X_{w_{k,l}}(b))$ for $w_{k,l}\in\SAdm(\mu)_{\DL}$, we obtain the following description of $X_\mu(b)$, which coincides with the conventional ones in fully Hodge-Newton cases $n\le 4$.
\begin{thm}[Theorem \ref{main theo}]
\label{stratification}
Let $w_{k,l}\in \SAdm(\mu)_{\neq \emptyset}$.
Then there exists a standard parahoric subgroup $P_{k,l}$, a Deligne-Lusztig variety $X_{k,l}$ and an irreducible component $Y_{k,l}$ of $\pi(X_{w_{k,l}}(b))$ such that $\pi(X_{w_{k,l}}(b))=\sqcup_{j\in G(F)/G(F)\cap P_{k,l}}jY_{k,l}$ and $Y_{k,l}$ is an iterated fibration over $X_{k,l}$.
In particular, the variety $X_\mu(b)$ is naturally a disjoint union of iterated fibrations over Deligne-Lusztig varieties.
\end{thm}
In Theorem \ref{main theo}, both $P_{k,l}$ and $X_{k,l}$ are described explicitly.
For example, if $w_{k,l}\in \SAdm(\mu)_{\neq \DL}$, then $P_{k,l}=G(\cO)$ and $X_{k,l}$ is a Deligne-Lusztig variety in a partial flag variety for $G(\aFq)$ determined by $w_{k,l}$ (see Remark \ref{DL U}).
Also, if $w_{k,l}$ is of positive Coxeter type (cf.\ \S\ref{LP}), then $Y_{k,l}$ is the product of a Deligne-Lusztig variety and a finite-dimensional affine space (see Corollary \ref{Coxeter coro}).

As a consequence of Theorem \ref{stratification}, we can describe the irreducible components of $X_\mu(b)$ as the closure of the irreducible components of top-dimensional Ekedahl-Oort strata (see Theorem \ref{irreducible components theo}).
In general, it is hard to describe $\overline{jY_{k,l}}=j\overline{Y_{k,l}}$ explicitly.
Indeed, if $n=5$, then the irreducible component $\overline{Y_{3,4}}$ cannot be written as a union of strata in Theorem \ref{stratification} (see Example \ref{counterexample}).
On the other hand, we have the following description for $w_{k,l}\in \SAdm(\mu)_{\DL}$, which is similar to the fully Hodge-Newton decomposable cases.

\begin{thm}[Corollary \ref{closure relation}]
Let $w_{k,l}\in \SAdm(\mu)_{\DL}$.
Then the closure of $jY_{k,l}(=jX_{k,l})$ in $X_\mu(b)$ is a union of strata in Theorem \ref{stratification}.
For $w_{k',l'}\in \SAdm(\mu)_{\neq \emptyset}$, $j\overline{Y_{k,l}}$ contains $j'Y_{k',l'}$ if and only if the following two conditions are both satisfied:
\begin{enumerate}[(i)]
\item $k\geq k'$ and $l\geq l'$.
\item $j(G(F)\cap P_{k,l})\cap j'(G(F)\cap P_{k',l'})\neq \emptyset$.
\end{enumerate}
In particular, $w_{k',l'}\in \SAdm(\mu)_{\DL}$ in this case.
\end{thm}
The proof is also similar to the fully Hodge-Newton decomposable cases.
The condition (i) is equivalent to both the Bruhat order $w_{k,l}\geq w_{k',l'}$ and a variant of it relative to the hyperspecial level.
This equivalence was already proved and conjectured in \cite[Proposition 3.4 \& Conjecture 3.23]{ABFGGN24}.
As explained in \cite[\S2.4]{GHN24}, we can express (ii) in terms of the building for $G(F)$.

Oort's conjecture asserts that, for $g\geq2$, generically on the supersingular locus of the moduli space of principally polarized abelian varieties of dimension $g$, the automorphism group is $\{\pm1\}$.
It was recently settled by Viehmann \cite{Viehmann26} using the Rapoport--Zink space for $\operatorname{GSp}_{2g}$.
She also proved an analogous result for the Rapoport--Zink space for $\GL_{2g}$ \cite[Theorem 5.1]{Viehmann26}. More recently, Philippe, Schnelle and Viehmann \cite{PSV26} proved analogous results for Rapoport--Zink spaces for $\GL_n$.
In the case of $\GL_n$, the Rapoport--Zink spaces describe the basic loci of unitary Shimura varieties at split primes.
Further generalizations for $\operatorname{GSO}_{4m}$ were obtained by the author and Takamatsu \cite{STOort}.
As an immediate application of Theorem \ref{stratification}, we also obtain the following, which is the first such result for a non-split group.
\begin{thm}[Theorem \ref{Oort analogue}]
Assume that $F=\Q_p$, $p\geq3$ and $n\geq3$.
There exists an open dense subscheme $U\subseteq X_\mu(b)$ such that, for every closed point $xK\in U$, every finite-order element of $G(\Q_p)$ fixing $xK$ is a unit-root scalar in $\Z_{p^2}^{\times}$.
\end{thm}

The paper is organized as follows.
In \S\ref{preliminaries} we introduce affine Deligne-Lusztig varieties and relevant notions.
The main ingredients of the proofs are the Deligne-Lusztig reduction method (Proposition \ref{DL method prop parahoric}) and a non-emptiness criterion of the Iwahori level affine Deligne-Lusztig varieties (Theorem \ref{empty}).
These techniques, which are completely different from the one used in \cite{FI21}, enable us to analyze affine Deligne-Lusztig varieties through combinatorics of the affine Weyl group.
In \S\ref{non-empty section}, we prove Theorem \ref{non-empty thm} by explicit computation in the affine Weyl group.
In \S\ref{geometric structure}, we prove the remaining main theorems by adding some supplementary computation.
In \S\ref{Oort analogue section}, we prove Theorem E.
The key point is that the fibrations in Theorem C are $G(\Q_p)$-equivariant, so the problem reduces to studying the usual action of $G(\F_p)$ on Deligne--Lusztig varieties.

\textbf{Acknowledgments:}
The author would like to thank Naoki Imai and Felix Schremmer for helpful comments.
The author became interested in Oort's conjecture thanks to Eva Viehmann.

This work was partially supported by the New Cornerstone Science Foundation through the New Cornerstone Investigator Program awarded to Professor Xuhua He.

\section{Preliminaries}
\label{preliminaries}
We will usually drop the adjective ``perfect'' for notational convenience even in the mixed characteristic case.
\subsection{Notation}
\label{notation}
Let $F$ be a non-archimedean local field with finite residue field $\F_q$ of prime characteristic $p$, and let $L$ be the completion of the maximal unramified extension of $F$.
Let $\sigma$ denote the Frobenius automorphism of $L/F$.
Further, we write $\cO$ (resp.\ $\cO_F$) for the valuation ring of $L$ (resp.\ $F$).
Finally, we denote by $\vp$ a uniformizer of $F$ and by $v_L$ the valuation of $L$ such that $v_L(\vp)=1$.

Let $G$ be an unramified connected reductive group over $\cO_F$.
Let $B\subset G$ be a Borel subgroup and $T\subset B$ a maximal torus in $B$, both defined over $\cO_F$.
For a cocharacter $\mu\in X_*(T)$, let $\vp^{\mu}$ be the image of $\vp\in \mathbb G_m(F)$ under the homomorphism $\mu\colon\mathbb G_m\rightarrow T$.

Let $\Phi=\Phi(G,T)$ denote the set of roots of $T$ in $G$.
We denote by $\Phi_+$ (resp.\ $\Phi_-$) the set of positive (resp.\ negative) roots distinguished by $B$.
Let $\Delta$ be the set of simple roots and $\Delta^\vee$ be the corresponding set of simple coroots.
Let $X_*(T)$ be the set of cocharacters, and let $X_*(T)_+$ be the set of dominant cocharacters.

The Iwahori-Weyl group $\tW=\tW_G$ is defined as the quotient $N_{G(L)}T(L)/T(\cO)$.
This can be identified with the semi-direct product $W_0\ltimes X_{*}(T)$, where $W_0$ is the finite Weyl group of $G$.
We denote the projection $\tW\rightarrow W_0$ by $p$.
We have a length function $\ell\colon \tW\rightarrow \Z_{\geq 0}$ given as
$$\ell(u\vp^{\lambda})=\sum_{\alpha\in \Phi_+, u\alpha\in \Phi_-}|\langle \alpha, \lambda\rangle+1|+\sum_{\alpha\in \Phi_+, u\alpha\in \Phi_+}|\langle \alpha, \lambda\rangle|,$$
where $u\in W_0$ and $\lambda\in \Y$.

Let $S\subset W_0$ denote the subset of simple reflections, and let $\tS\subset \tW$ denote the subset of simple affine reflections.
We often identify $\Delta$ and $S$.
The affine Weyl group $W_a$ is the subgroup of $\tW$ generated by $\tS$.
Then we can write the Iwahori-Weyl group as a semi-direct product $\tW=W_a\rtimes \Omega$, where $\Omega\subset \tW$ is the subgroup of length $0$ elements.
Moreover, $(W_a, \tS)$ is a Coxeter system.
We denote by $\le$ the Bruhat order on $\tW$.
For any $J\subseteq \tS$, let $^J\tW$ (resp.\ $\tW^J$) be the set of minimal length elements for the cosets in $W_J\backslash \tW$ (resp.\ $\tW/W_J$), where $W_J$ denotes the subgroup of $\tW$ generated by $J$.
We write $^J\tW^{J'}$ for $^J\tW\cap \tW^{J'}$.

For $w\in W_a$, we denote by $\supp(w)\subseteq \tS$ the set of simple affine reflections occurring in every (equivalently, some) reduced expression of $w$.
Note that $\tau\in \Omega$ acts on $\tS$ by conjugation.
We define the $\sigma$-support $\supp_\sigma(w\tau)$ of $w\tau$ as the smallest $\tau\sigma$-stable subset of $\tS$ which contains $\supp(w)$.
We call an element $w\tau\in W_a\tau$ a $\sigma$-Coxeter element if exactly one simple reflection from each $\tau\sigma$-orbit on $\supp_\sigma(w\tau)$ occurs in every (equivalently, any) reduced expression of $w$.

For $w,w'\in \tW$ and $s\in \tS$, we write $w\xrightarrow{s} w'$ if $w'=sw\sigma(s)$ and $\ell(w')\le \ell(w)$.
We write $w\rightarrow w'$ if there is a sequence $w=w_0,w_1,\ldots, w_k=w'$ of elements in $\tW$ such that for any $i$, $w_{i-1}\xrightarrow{s_i} w_i$ for some $s_i\in \tS$.
If $\ell(w)=\ell(w')$, we write $w\approx w'$.
We also write $w\xrightarrow{s_k\cdots s_1s_0} w'$ when we want to specify the sequence defining $w\rightarrow w'$.
Note that in many other papers, $\rightarrow$ and $\approx$ are denoted as $\rightarrow_\sigma$ and $\approx_\sigma$ respectively.
We drop $\sigma$ for notational convenience.

For $\alpha\in \Phi$, let $U_\alpha\subseteq G$ denote the corresponding root subgroup.
We set $$I=T(\cO)\prod_{\alpha\in \Phi_+}U_{\alpha}(\vp\cO)\prod_{\beta\in \Phi_-}U_{\beta}(\cO)\subseteq G(L),$$
which is called the standard Iwahori subgroup associated to $T\subset B\subset G$.
For $J\subset \tS$ with $W_J$ finite, let $P_J\supset I$ be the standard parahoric subgroup associated to $J$.
We denote by $\pi_J$ the projection $G(L)/I\rightarrow G(L)/P_J$.
Set $K=P_S=G(\cO)$ and $\pi=\pi_S$.

In the case $G=\GL_n$, we will use the following description.
Let $T$ be the torus of diagonal matrices, and we choose the subgroup of upper triangular matrices $B$ as Borel subgroup.
Let $\chi_{ij}$ be the character $T\rightarrow \Gm$ defined by $\mathrm{diag}(t_1,t_2,\ldots, t_n)\mapsto t_i{t_j}^{-1}$.
Then we have $\Phi=\{\chi_{ij}\mid i\neq j\}$, $\Phi_+=\{\chi_{ij}\mid i< j\}$, $\Phi_-=\{\chi_{ij}\mid i> j\}$ and $\Delta=\{\chi_{i,i+1}\mid 1\le i <n\}$.
Through the isomorphism $X_*(T)\cong \Z^n$, ${X_*(T)}_+$ can be identified with the set $\{(m_1,\cdots, m_n)\in \Z^n\mid m_1\geq \cdots \geq m_n\}$.
Let us write $s_1=(1\ 2), s_2=(2\ 3), \ldots, s_{n-1}=(n-1\ n)$.
Set $s_0=\vp^{\chi_{1,n}^{\vee}}(1\ n)$, where $\chi_{1,n}$ is the unique highest root.
Then $S=\{s_1,s_2,\ldots, s_{n-1}\}$ and $\tS=S\cup\{s_0\}$.
The Iwahori subgroup $I\subset K$ is the inverse image of $B^{\op}$ under the projection $G(\cO)\rightarrow G(\aFq)$ sending $\vp$ to $0$, where $B^{\op}$ is the subgroup of lower triangular matrices.
Similarly, if $J\subset S$, then $P_J$ is the inverse image of the standard parabolic subgroup (which contains $B^{\op}$) corresponding to $J$.
Finally, $\vp^{(1,0^{(n-1)})}s_1s_2\cdots s_{n-1}$ is a generator of $\Omega\cong \Z$.

\subsection{Affine Deligne-Lusztig Varieties}
\label{ADLV}
For $w\in \tW$ and $b\in G(L)$, the affine Deligne-Lusztig variety $X_w(b)$ in the affine flag variety $G(L)/I$ is defined as
$$X_w(b)=\{xI\in G(L)/I\mid x^{-1}b\sigma(x)\in IwI\}.$$
For $\mu\in \Y_+$ and $b\in G(L)$, the affine Deligne-Lusztig variety $X_{\mu}(b)$ in the affine Grassmannian $G(L)/K$ is defined as
$$X_{\mu}(b)=\{xK\in G(L)/K\mid x^{-1}b\sigma(x)\in K\vp^{\mu}K\}.$$
For $w\in {^J\tW^{\sigma(J)}}$ and $b\in G(L)$, the (coarse) affine Deligne-Lusztig variety $X_{J,w}(b)$ in $G(L)/P_J$ is defined as
$$X_{J,w}(b)=\{gP_J\in G(L)/P_J\mid g^{-1}b\sigma(g)\in P_JwP_{\sigma(J)}\}.$$
In the equal characteristic case, affine Deligne-Lusztig varieties are schemes, locally of finite type over $\aFq$.
In the mixed characteristic case, affine Deligne-Lusztig varieties are perfect schemes, locally perfectly of finite type over $\aFq$.
See \cite{PR08}, \cite{Zhu17}, \cite{BS17} and \cite[Lemma 1.1]{HV18}.
Left multiplication by $g^{-1}\in G(L)$ induces an isomorphism between affine Deligne-Lusztig varieties corresponding to $b$ and $g^{-1}b\sigma(g)$.
Thus the isomorphism class of the affine Deligne-Lusztig variety only depends on the $\sigma$-conjugacy class of $b$.
Also, the affine Deligne-Lusztig varieties carry a natural action (by left multiplication) by the $\sigma$-centralizer of $b$
$$\J_b=\{g\in G(L)\mid g^{-1}b\sigma(g)=b\}.$$
Note that $\J_b\cong \J_{g^{-1}b\sigma(g)}$ by sending $j$ to $g^{-1}jg$.
\begin{rema}
In \cite[\S 3.4]{GH15}, $\pi_J(X_w(b))$ was denoted by $X_{J,w}(b)$.
We hope our notation will not cause confusions.
\end{rema}

The admissible subset of $\tW$ associated to $\mu$ is defined as
$$\Adm(\mu)=\{w\in \tW\mid w\le \vp^{u\mu}\ \text{for some}\ u\in W_0\}.$$
Set $\SAdm(\mu)=\Adm(\mu)\cap \SW$.
Assume that $\mu$ is minuscule.
Then, by \cite[Theorem 3.2.1]{GH15} (see also \cite[\S2.5]{GHR20}), we have
$$X_{\mu}(b)=\bigsqcup_{w\in\SAdm(\mu)}\pi(X_w(b)).$$
This is called the {\it Ekedahl-Oort stratification}.
Note that
$$\pi(X_w(b))=\{gK\in G(L)/K\mid g^{-1}b\sigma(g)\in K\cdot_\sigma IwI\},$$
where $\cdot_\sigma$ denotes the $\sigma$-twisted conjugation action of $G(L)$.

Set $S(w,\sigma)=\max\{S'\subseteq S\mid \Ad(w)\sigma(S')=S'\}$.
Then $P_{S(w,\sigma)}wP_{\sigma (S(w,\sigma))}=P_{S(w,\sigma)}\cdot_\sigma IwI$ for $w\in \SW$ by \cite[Theorem 3.2.1]{GH15}.
Moreover, by \cite[Theorem 4.1.2 (1)]{GH15}, the projection $G(L)/P_{S(w,\sigma)}\rightarrow G(L)/K$ induces an isomorphism $\pi_{S(w,\sigma)}(X_w(b))=X_{S(w,\sigma),w}(b)\cong\pi(X_w(b))$.

It follows from \cite[Proposition 5.7]{GHN19} that if $W_{\supp_\sigma(w)}$ is finite, then $W_{\supp_\sigma(w)\cup S(w,\sigma)}$ is also finite.
The following proposition is a combination of \cite[Proposition 2.2.1 \& \S4.1]{GH15} (see also \cite[\S 2.4]{GHN24} and \cite[Lemma 2.5]{Gortz19}).
\begin{prop}
\label{spherical}
Let $\tau\in \Omega$.
Let $w\in W_a\tau$ such that $W_{\supp_\sigma(w)}$ is finite.
Then $$X_w(\tau)=\bigsqcup_{j\in \J_\tau/\J_\tau\cap P_{\supp_\sigma(w)}} jY(w),$$
where $Y(w)=\{gI\in P_{\supp_\sigma(w)}/I\mid g^{-1}\tau \sigma(g)\in IwI\}$ is a Deligne-Lusztig variety in the flag variety $P_{\supp_\sigma(w)}/I$.

Let $w\in \SW \cap W_a\tau$ such that $W_{\supp_\sigma(w)}$ is finite.
Then
$$\pi(X_w(\tau))\cong\pi_{S(w,\sigma)}(X_w(\tau))=\bigsqcup_{j\in \J_\tau/\J_\tau\cap P_{\supp_\sigma(w)\cup S(w,\sigma)}} j\pi_{S(w,\sigma)}(Y(w)),$$
where $\pi_{S(w,\sigma)}(Y(w))=\{gP_{S(w,\sigma)}\in P_{\supp_\sigma(w)\cup S(w,\sigma)}/P_{S(w,\sigma)}\mid g^{-1}\tau \sigma(g)\in P_{S(w,\sigma)}\cdot_\sigma IwI\}$ is a Deligne-Lusztig variety in the partial flag variety $P_{\supp_\sigma(w)\cup S(w,\sigma)}/P_{S(w,\sigma)}$.
\end{prop}

\begin{rema}
\label{spherical rema}
Each $Y(w)$ or $\pi_{S(w,\sigma)}(Y(w))$ is irreducible and of dimension $\ell(w)$.
See \cite{DL76} and \cite{BR06}.
Clearly, each $jY(w)$ or $j\pi_{S(w,\sigma)}(Y(w))$ is closed in $X_w(\tau)$ or $\pi(X_w(\tau))$, and hence an irreducible component.
\end{rema}

\subsection{Deligne-Lusztig Reduction Method}
\label{DL method}
The following Deligne-Lusztig reduction method was established in \cite[Corollary 2.5.3]{GH10} ($\A^1$ and $\G_m$ actually mean $\A^{1,\pfn}$ and $\G_m^{\pfn}$ respectively in the mixed characteristic case).
\begin{prop}
\label{DL method prop}
Let $w\in \tW$ and let $s\in \tS$ be a simple affine reflection.
Then the following two statements hold for any $b\in G(L)$.
\begin{enumerate}[(i)]
\item If $\ell(sw\sigma(s))=\ell(w)$, then there exists a $\J_b$-equivariant universal homeomorphism $X_w(b)\rightarrow X_{sw\sigma(s)}(b)$.
\item If $\ell(sw\sigma(s))=\ell(w)-2$, then there exists a decomposition $X_w(b)=X_1\sqcup X_2$ such that
\begin{itemize}
\item $X_1$ is open and there exists a $\J_b$-equivariant morphism $X_1\rightarrow X_{sw}(b)$, which is  the composition of a Zariski-locally trivial $\G_m$-bundle and a universal homeomorphism. 
\item $X_2$ is closed and there exists a $\J_b$-equivariant morphism $X_2\rightarrow X_{sw\sigma(s)}(b)$, which is the composition of a Zariski-locally trivial $\A^1$-bundle and a universal homeomorphism. 
\end{itemize}
\end{enumerate}
\end{prop}

Let $gI\in X_w(b)$.
If $\ell(sw)<\ell(w)$ (we can reduce to this case by exchanging $w$ and $sw\sigma(s)$), then let $g_1I$ denote the unique element in $G(L)/I$ such that $g^{-1}g_1\in IsI$ and $g_1^{-1}b\sigma(g)\in IswI$.
The set $X_1$ (resp.\ $X_2$) above consists of the elements $gI\in X_w(b)$ satisfying $g_1^{-1}b\sigma(g_1)\in IswI$ (resp.\ $Isw\sigma(s)I$).
All of the maps in the proposition are given as the map sending $gI$ to $g_1I$.

The following lemma is easy to prove (cf.\ \cite[Lemma 5.1]{Lansky01} and \cite[Proposition 3.16 (ii)]{BT72}).
\begin{lemm}
\label{parahoric}
Let $J\subset \tS$, $w\in {^J\tW}$ and $s\in \tS\setminus J$.
\begin{enumerate}[(i)]
\item If $s$ and $J$ commute (i.e., $ss_j=s_js$ for any $j\in J$), then $P_J sP_J=IsP_J=P_J sI$.
Moreover, the projection induces an isomorphism $IsI/I\xrightarrow{\sim} IsP_J/P_J$.
\item If $\Ad(w)\sigma(J)=J$ (and hence $w\in {^J\tW^{\sigma(J)}}$), then $P_J wP_{\sigma(J)}=IwP_{\sigma(J)}$.
Moreover, the projection induces an isomorphism $IwI/I\xrightarrow{\sim} IwP_{\sigma(J)}/P_{\sigma(J)}$.
\end{enumerate}
\end{lemm}

Assume that $w\in {^J\tW}$ and $\Ad(w)\sigma(J)=J$.
Let $s\in \tS\setminus J$ such that $s$ and $J$ commute.
If $\ell(sw)<\ell(w)$, then by Lemma \ref{parahoric}, the projection $(G(L)/P_J)^2\times (G(L)/P_{\sigma(J)})\rightarrow (G(L)/P_J)\times (G(L)/P_{\sigma(J)})$ to the first and third factors induces an isomorphism of varieties 
\begin{align*}
&\{(gP_J,g''P_J,g'P_{\sigma(J)})\mid g^{-1}g''\in P_JsP_J,g''^{-1}g'\in P_J swP_{\sigma (J)}\}\\\xrightarrow{\sim}&\{(gP_J,g'P_{\sigma(J)})\mid g^{-1}g'\in P_J wP_{\sigma (J)}\}.
\end{align*}
Indeed, $P_JsP_JswP_{\sigma (J)}=P_JsIswP_{\sigma (J)}=P_J wP_{\sigma(J)}$ and
$$P_J\cap(sP_Js)\cap (wP_{\sigma(J)}w^{-1})=P_J\cap (wP_{\sigma(J)}w^{-1}).$$
This isomorphism restricts to an isomorphism $$\{(gP_J,g_1P_J,b\sigma(g)P_{\sigma(J)})\mid g^{-1}g_1\in P_JsP_J,g_1^{-1}b\sigma(g)\in P_J swP_{\sigma (J)}\}\xrightarrow{\sim} X_{J,w}(b).$$
Again by Lemma \ref{parahoric}, we have
\begin{align*}
P_{J} swP_{\sigma(J)}\sigma(s)P_{\sigma(J)}&=P_JswI\sigma(s)P_{\sigma(J)}\\
&=
\begin{cases}
P_Jsw\sigma(s)P_{\sigma(J)} & \text{if}\ \ell(sw\sigma(s))>\ell(sw)\\
P_JswP_{\sigma(J)}\sqcup P_Jsw\sigma(s)P_{\sigma(J)} & \text{if}\ \ell(sw\sigma(s))<\ell(sw).
\end{cases}
\end{align*}
So the same proof as \cite[Corollary 2.5.3]{GH10} shows the following proposition.
\begin{prop}
\label{DL method prop parahoric}
Let $w\in {^J\tW}$ with $\Ad(w)\sigma(J)=J$.
Let $s\in \tS\setminus J$ such that $s$ and $J$ commute.
Then $sw,sw\sigma(s)\in {^J\tW}$, $\Ad(sw)\sigma(J)=J$ and $\Ad(sw\sigma(s))\sigma(J)=J$.
Moreover, the following two statements hold for any $b\in G(L)$.
\begin{enumerate}[(i)]
\item If $\ell(sw\sigma(s))=\ell(w)$, then there exists a $\J_b$-equivariant universal homeomorphism $X_{J,w}(b)\rightarrow X_{J,sw\sigma(s)}(b)$.
\item If $\ell(sw\sigma(s))=\ell(w)-2$, then there exists a decomposition $X_{J,w}(b)=X_1\sqcup X_2$ such that
\begin{itemize}
\item $X_1$ is open and there exists a $\J_b$-equivariant morphism $X_1\rightarrow X_{J,sw}(b)$, which is  the composition of a Zariski-locally trivial $\G_m$-bundle and a universal homeomorphism. 
\item $X_2$ is closed and there exists a $\J_b$-equivariant morphism $X_2\rightarrow X_{J,sw\sigma(s)}(b)$, which is the composition of a Zariski-locally trivial $\A^1$-bundle and a universal homeomorphism. 
\end{itemize}
\end{enumerate}
\end{prop}

Let $gP_J\in X_{J,w}(b)$.
If $\ell(sw)<\ell(w)$, then all of the maps in the proposition are given as the map sending $gP_J$ to $g_1P_J$.
In \S\ref{geometric structure}, we will use Proposition \ref{DL method prop parahoric} only in the situation where $J=S(w,\sigma)$.

We say that a scheme $X$ is an {\it iterated fibration} of rank $a$ over a scheme $Y$ (whose fibers are all $\A^1$) if there exist morphisms
$$X=Y_0\rightarrow Y_1\rightarrow \cdots \rightarrow Y_a=Y$$
such that $Y_{i}$ is a Zariski-locally trivial $\A^1$-bundle over $Y_{i+1}$ for any $0\le i<a$.

\begin{rema}
The perfection of a universal homeomorphism is an isomorphism.
\end{rema}

\subsection{Length Positive Elements}
\label{LP}
We denote by $\delta^+$ the indicator function of the set of positive roots, i.e.,
$$\delta^+\colon \Phi\rightarrow \{0,1\},\quad \alpha \mapsto
\begin{cases}
1 & (\alpha\in \Phi_+) \\
0 & (\alpha\in \Phi_-).
\end{cases}$$
Note that any element $w\in \tW$ can be written in a unique way as $w=x\vp^\mu y$ with $\mu$ dominant, $x,y\in W_0$ such that $\vp^\mu y\in \SW$.
We have $p(w)=xy$ and $\ell(w)=\ell(x)+\la\mu, 2\rho\ra-\ell(y)$.
We define the set of {\it length positive} elements by $$\LP(w)=\{v\in W_0\mid \la v\alpha,y^{-1}\mu\ra+\delta^+(v\alpha)-\delta^+(xyv\alpha)\geq 0\  \text{for all $\alpha\in \Phi_+$}\}.$$
Then we always have $y^{-1}\in \LP(w)$.
Indeed, $y$ is uniquely determined by the condition that
$\la\alpha, \mu\ra\geq \delta^+(-y^{-1}\alpha)\ \text{for all $\alpha\in \Phi_+$}$.
Since $\delta^+(\alpha)+\delta^+(-\alpha)=1$, we have $$\la y^{-1}\alpha, y^{-1}\mu\ra+\delta^+(y^{-1}\alpha)-\delta^+(x\alpha)=\la \alpha,\mu\ra-\delta^+(-y^{-1}\alpha)+\delta^+(-x\alpha)\geq 0.$$

For any $w=x\vp^\mu y\in \tW$ as above, we define $$\Phi_w\coloneqq\{\alpha\in \Phi_+\mid \la\alpha,\mu\ra-\delta^-(y^{-1}\alpha)+\delta^-(x\alpha)=0\}.$$
Here $\delta^-$ denotes the indicator function of the set of negative roots. Set
$$R(w)=\{r^{-1}\in W_0\mid r(\Phi_+\setminus \Phi_w)\subset \Phi_+\ \text{or equivalently,}\ r^{-1}\Phi_+\subset \Phi_+\cup-\Phi_w\}.$$
Then we have $y\LP(w)=R(w)$.
See \cite[Lemma 2.8]{Shimada4}.

Thanks to Kottwitz \cite{Kottwitz85}, a $\sigma$-conjugacy class $[b]$ of $b\in G(L)$ is uniquely determined by two invariants: the Kottwitz point $\kappa(b)\in \pi_1(G)/((1-\sigma)\pi_1(G))$ and the Newton point $\nu_b\in X_*(T)_{\Q,+}$.
Clearly, $X_w(b)=\emptyset$ if $\kappa(b)\neq\kappa(\dot w)$.
We say that $b\in G(L)$ is basic if $\nu_b$ is central.
The following theorem is a refinement of the non-emptiness criterion in \cite{GHN15}, which is conjectured by Lim in \cite{Lim23} and proved by Schremmer in \cite[Proposition 6]{Schremmer23}.
\begin{theo}
\label{empty}
Assume that the Dynkin diagram of $G$ is $\sigma$-connected, i.e., $\sigma$ acts transitively on the set of irreducible components of $\Phi$.
Let $w\in \tW$ and let $b\in G(L)$ be a basic element with $\kappa(b)=\kappa(\dot w)$.
Then $X_w(b)=\emptyset$ if and only if the following two conditions are satisfied:
\begin{enumerate}[(i)]
\item $|W_{\supp_\sigma(w)}|$ is infinite.
\item There exists $v\in \LP(w)$ such that $\supp_\sigma(\sigma^{-1}(v)^{-1}p(w)v)\subsetneq S$.
\end{enumerate}
\end{theo}

We say that $w$ is of {\it positive Coxeter type} if $\sigma^{-1}(v)^{-1}p(w)v$ is a $\sigma$-Coxeter element for some $v\in \LP(w)$.
Affine Deligne-Lusztig varieties associated to elements of positive Coxeter type were studied in a joint work \cite{SSY23} with Schremmer and Yu.

We will use Theorem \ref{empty} for $\vp^\mu y\in \SW$ as above.
In this case, the condition (ii) is equivalent to saying that there exists $r^{-1}\in R(\vp^\mu y)$ such that $\supp_\sigma(ry\sigma(r)^{-1})\subsetneq S$.
In particular, (ii) is satisfied if $\supp_\sigma(y)\subsetneq S$.

\section{Non-Empty Ekedahl-Oort Strata}
\label{non-empty section}
We assume that $n\geq 2$.
\subsection{Setting}
\label{setting}
Let $F_2$ be the quadratic unramified extension of $F$.
Let $\cO_{F_2}$ denote the ring of integers of $F_2$.
We put $\Lambda=\cO_{F_2}^n$ equipped with the hermitian form
$$\Lambda\times \Lambda\rightarrow \cO_{F_2},\quad ((a_i)_{1\le i\le n},(b_i)_{1\le i\le n})\mapsto \sum_{i=1}^{n}\sigma(a_i)b_{n+1-i}.$$
From now on, we set $G=\mathrm{GU}(\Lambda)$.
By taking the first factor of the isomorphism
$$\cO_{F_2}\otimes_{\cO_F}\cO_{F_2}\cong \cO_{F_2}\times \cO_{F_2},\quad a\otimes b\mapsto (ab,a\sigma(b)),$$
we obtain an isomorphism $G_{\cO_{F_2}}\cong \GL_n\times \G_m$.
Let $T\subset B\subset G$ be the maximal torus and the Borel subgroup determined by the diagonal torus and the upper triangular subgroup of $\GL_n$ under this isomorphism.
Then $\Y$ can be identified with $\Z^n\times \Z$ and $\tW=\tW_G\cong \tW_{\GL_n}\times \Z$.
Under this identification, we have $\sigma(((m_i)_{1\le i\le n},0))=((-m_{n+1-i})_{1\le i\le n},0)\in \Y$ and $\sigma((s_i,0))=(s_{n-i},0)\in \tW$ by setting $s_n=s_0$.

In the sequel, we set $\mu=((0^{(n-2)},-1,-1),-1)$ and $b=\vp^{((0^{(n)}),-1)}$.
Then $b$ is a central basic element with $\kappa(b)=\kappa(\vp^\mu)$ (cf.\ \cite[\S 5.2]{TY22}).
By abuse of notation, we write $s_i$ for $(s_i,0)\in \tW$.
For integers $0\le k, l\le n$, we set 
$$s_{[k,l]}=\begin{cases}
s_ks_{k-1}\cdots s_l & \text{if $k\geq l$} \\
1 & \text{otherwise}
\end{cases}$$
and $w_{k,l}=\vp^\mu s_{[n-2,k]}s_{[n-1,l]}$.
Then $\SAdm(\mu)=\{w_{k,l}\mid 1\le k<l\le n\}$.
In particular, $\tau\coloneqq w_{1,2}$ is the length $0$ element corresponding to $\mu$.
It is straightforward to check that $\ell(w_{k,l})=k+l-3$ and $$\Phi_+\setminus \Phi_{w_{k,l}}=\{\chi_{1,n-1},\chi_{2,n-1},\ldots,\chi_{k-1,n-1},\chi_{1,n},\chi_{2,n},\ldots,\chi_{l-2,n}\}.$$

Set $\tau_1\coloneqq \vp^{((1,0^{(n-1)}),0)}s_1s_2\cdots s_{n-1}\in \Omega$.
Then $\tau_1^{-1}b\sigma(\tau_1)=\tau$.
We define $\SAdm(\mu)_{\neq \emptyset}\coloneqq\{w\in \SAdm(\mu)\mid X_w(b)\neq \emptyset\}$ and $\SAdm(\mu)_{\DL}\coloneqq\{w\in \SAdm(\mu)\mid \supp_\sigma(w)\neq \tS\}$.
By Theorem \ref{empty}, we have $\SAdm(\mu)_{\DL}\subseteq \SAdm(\mu)_{\neq \emptyset}$.
We denote $\SAdm(\mu)_{\neq \emptyset}\setminus\SAdm(\mu)_{\DL}$ by $\SAdm(\mu)_{\neq\DL}$.

\subsection{Computation in the Affine Weyl Group}
If $n$ is odd, then $\sigma$-orbits on $\tS$ are
$$\{s_0\},\{s_1,s_{n-1}\},\{s_2,s_{n-2}\},\ldots, \{s_{\frac{n-1}{2}}, s_{\frac{n+1}{2}}\}.$$
Since $\tau s_i \tau^{-1}=s_{i-2}$, $\tau\sigma$-orbits on $\tS$ are
$$\{s_{n-1}\},\{s_0,s_{n-2}\},\{s_1,s_{n-3}\},\ldots, \{s_{\frac{n-3}{2}}, s_{\frac{n-1}{2}}\}.$$
If $n$ is even, then $\sigma$-orbits on $\tS$ are
$$\{s_0\},\{s_1,s_{n-1}\},\{s_2,s_{n-2}\},\ldots, \{s_{\frac{n}{2}-1}, s_{\frac{n}{2}+1}\}, \{s_{\frac{n}{2}}\}.$$
Since $\tau s_i \tau^{-1}=s_{i-2}$, $\tau\sigma$-orbits on $\tS$ are
$$\{s_{n-1}\},\{s_0,s_{n-2}\},\{s_1,s_{n-3}\},\ldots, \{s_{\frac{n}{2}-2}, s_{\frac{n}{2}}\},\{s_{\frac{n}{2}-1}\}.$$
Set $t_i=s_is_{n-i}(=s_{n-i}s_i$ unless $\frac{n-1}{2}\le i\le \frac{n+1}{2}$).
We simply write $\equiv$ for $\equiv$ (mod $2$).
\begin{lemm}
\label{empty lemm}
Assume that $1\le k<l\le n$ satisfies one of the following conditions:
\begin{enumerate}[(i)]
\item $\frac{n+2}{2}\le k<l\le n$.
\item $l\equiv n$, $\frac{n+3}{2}\le l$ and $n-l+2\le k\le \frac{n+1}{2}$.
\item $k$ is even, $k\le \frac{n-1}{2}$ and $\frac{n+3}{2}\le l\le n-k+2$.
\end{enumerate}
Then $X_{w_{k,l}}(b)=\emptyset$.
\end{lemm}
\begin{proof}
Note that $w_{k,l}=\tau (s_2s_3\cdots s_{l-1})(s_1s_2\cdots s_{k-1})=(s_0s_1\cdots s_{l-3})(s_{n-1}s_0\cdots s_{k-3})\tau$.
It is straightforward to check that if $\frac{n+2}{2}\le k<l\le n$, then $\supp_\sigma(w_{k,l})=\tS$ and $\supp_\sigma(s_{[n-2,k]}s_{[n-1,l]})\subsetneq S$.
Thus if (i) holds, then $X_{w_{k,l}}(b)=\emptyset$ by Theorem \ref{empty}.

Assume that (ii) holds.
Then by $\frac{n+3}{2}\le l\le n$ and hence $k\geq 2$, we have $\supp_\sigma(w_{k,l})=\tS$.
We define $r\coloneqq(t_{\frac{n-l}{2}+1})\cdots (t_3t_4\cdots t_{n-l-1})(t_2t_3\cdots t_{n-l})$ ($r=1$ if $l=n$).
Then $r(n)=n, r(n-1)=\frac{n+l}{2}-1$ and $r\{1,2,\ldots, k-1\}=\{1,2,\ldots, k-1\}$ by $n-l+2\le k$.
So we have $r^{-1}\in R(w_{k,l})$ by $k\le \frac{n+1}{2}< \frac{n+l}{2}$.
It follows from $n-l+2\le k$ that
\begin{align*}
&(t_2t_3\cdots t_{n-l})s_{[n-2,k]}s_{[n-1,l]}\sigma(t_2t_3\cdots t_{n-l})^{-1}=s_{[n-2,l]}s_{[n-2,k]}s_{n-1}.
\end{align*}
It follows from $s_{[n-2,l]}^2=(s_{n-3}s_{n-2})(s_{n-4}s_{n-3})\cdots (s_ls_{l+1})=s_{[n-3,l]}s_{[n-2,l+1]}$ and $n-l+2\le k$ that
\begin{align*}
&(t_3t_4\cdots t_{n-l-1})s_{[n-2,l]}s_{[n-2,k]}s_{n-1}\sigma(t_3t_4\cdots t_{n-l-1})^{-1}=s_{[n-3,l+1]}s_{[n-3,k]}s_{[n-1,n-2]}^{-1}.
\end{align*}
Similarly, for $1\le i\le \frac{n-l}{2}-1$, we have $s_{[n-1-i, l-1+i]}^2=s_{[n-2-i,l-1+i]}s_{[n-1-i,l+i]}$ and
\begin{align*}
&(t_{2+i}\cdots t_{n-l-i})s_{[n-1-i, l-1+i]}s_{[n-1-i,k]}s_{[n-1,n-i]}^{-1}\sigma(t_{2+i}\cdots t_{n-l-i})^{-1}\\
=&(t_{2+i}\cdots t_{n-l-i})s_{[n-1-i, l-1+i]}^2 \sigma(t_{2+i}\cdots t_{n-l-i})^{-1}s_{[l-2+i,k]}s_{[n-1,n-i]}^{-1}\\
=&s_{[n-2-i, l+i]}s_{[n-2-i,k]}s_{[n-1,n-1-i]}^{-1}.
\end{align*}
Hence we deduce that $$rs_{[n-2,k]}s_{[n-1,l]}\sigma(r)^{-1}=s_{\frac{n+l}{2}-1}s_{[\frac{n+l}{2}-1,k]}s_{[n-1,\frac{n+l}{2}]}^{-1}=s_{[\frac{n+l}{2}-2,k]}s_{[n-1,\frac{n+l}{2}]}^{-1}.$$
Thus $\supp_\sigma(rs_{[n-2,k]}s_{[n-1,l]}\sigma(r)^{-1})\subsetneq S$ by $\frac{n+1}{2}\le \frac{n+l}{2}-1$ and $\frac{n-l}{2}+1<n-l+2\le k$.
Therefore $X_{w_{k,l}}(b)=\emptyset$ by Theorem \ref{empty}.

Assume that (iii) holds.
Then by $l\geq\frac{n+3}{2}$ and $k\geq 2$, we have $\supp_\sigma(w_{k,l})=\tS$.
If $k=2$, then $s_{n-2}\in R(w_{2,l})$, $\supp_\sigma(s_{n-2}s_{[n-2,2]}s_{[n-1,l]}\sigma(s_{n-2})^{-1})\subsetneq S$ and hence $X_{w_{k,l}}(b)=\emptyset$ by Theorem \ref{empty}.
So we may assume that $k\geq 4$.
We define 
\begin{align*}
r'&\coloneqq (t_{l}\cdots t_{l+k-6}t_{l+k-4})\cdots s_{n-k-1}(t_{n-k+1}\cdots t_{n-5}t_{n-3})s_{n-k}(t_{n-k+2}\cdots t_{n-4}t_{n-2})\\
r&\coloneqq (t_{l+\frac{k}{2}-2})\cdots (t_{l+2}\cdots t_{l+k-8}t_{l+k-6})(t_{l+1}\cdots t_{l+k-7}t_{l+k-5})r'\quad (k\geq 6)
\end{align*}
for even $k$ ($r'=t_{n-k+2}\cdots t_{n-4}t_{n-2}$ if $l=n-k+2$).
Then 
\begin{align*}
&r'(n)=r(n)=n,\quad r'(n-1)=l+k-4,\quad r(n-1)=l+\tfrac{k}{2}-2,\\
&r'\{1,2,\ldots, l-2\}=r\{1,2,\ldots, l-2\}=\{1,2,\ldots, l-2\}.
\end{align*}
Indeed, both $\supp(r')$ and $\supp(r)$ do not contain $s_{l-2}$ by $n-l\le \frac{n-3}{2}<l-2$.
So we have $r'^{-1},r^{-1}\in R(w_{k,l})$.
It follows from $4\le k\le \frac{n-1}{2}$ and $\frac{n+3}{2}\le l\le n-k+2$ that
\begin{align*}
&(t_{n-k+2}\cdots t_{n-4}t_{n-2})s_{[n-2,k]}s_{[n-1,l]}\sigma(t_{n-k+2}\cdots t_{n-4}t_{n-2})^{-1}\\
=&(t_{n-k+2}\cdots t_{n-4})s_{[n-3,k]}s_{[n-1,l]}s_{n-2}\sigma(t_{n-k+2}\cdots t_{n-4})^{-1}\\
=&(t_{n-k+2}\cdots t_{n-4})s_{n-4}s_{[n-3,k]}s_{[n-1,l]}\sigma(t_{n-k+2}\cdots t_{n-4})^{-1}\\
=&\cdots\\
=&s_{[n-3,k]}s_{[n-1,l]}s_{n-k+2}.
\end{align*}
If $l<n-k+2$, then $s_{n-k}s_{[n-3,k]}s_{[n-1,l]}s_{n-k+2}\sigma(s_{n-k})=s_{[n-3,k+1]}s_{[n-1,l]}$.
Similarly, for $0\le i\le n-l-k+2$ (note that $k+i\le n-l+2<l$), we have
\begin{align*}
&(t_{n-k+2-i}\cdots t_{n-4-i}t_{n-2-i})s_{[n-2-i, k+i]}s_{[n-1,l]}\sigma(t_{n-k+2-i}\cdots t_{n-4-i}t_{n-2-i})^{-1}\\
=&(t_{n-k+2-i}\cdots t_{n-4-i})s_{[n-3-i,k+i]}s_{[n-1,l]}s_{n-2-i}\sigma(t_{n-k+2-i}\cdots t_{n-4-i})^{-1}\\
=&(t_{n-k+2-i}\cdots t_{n-4-i})s_{n-4-i}s_{[n-3-i,k+i]}s_{[n-1,l]}\sigma(t_{n-k+2-i}\cdots t_{n-4-i})^{-1}\\
=&\cdots\\
=&s_{[n-3-i,k+i]}s_{[n-1,l]}s_{n-k+2-i}.
\end{align*}
If $l<n-k+2-i$, then $s_{n-k-i}s_{[n-3-i,k+i]}s_{[n-1,l]}s_{n-k+2-i}\sigma(s_{n-k-i})=s_{[n-3-i,k+i+1]}s_{[n-1,l]}$.
Thus $r's_{[n-2,k]}s_{[n-1,l]}\sigma(r')^{-1}=s_{[l+k-5,n-l+2]}s_{[n-1,l+1]}$.
This implies that if $k=4$, then $\supp_\sigma(r's_{[n-2,k]}s_{[n-1,l]}\sigma(r')^{-1})\subsetneq S$ and hence $X_{w_{k,l}}(b)=\emptyset$ by Theorem \ref{empty}.
If $k\geq 6$, then it is easy to check that $rs_{[n-2,k]}s_{[n-1,l]}\sigma(r)^{-1}=s_{[l+\frac{k}{2}-3,n-l+2]}s_{[n-1,l+\frac{k}{2}-1]}$.
Thus $\supp_\sigma(rs_{[n-2,k]}s_{[n-1,l]}\sigma(r)^{-1})\subsetneq S$ and hence $X_{w_{k,l}}(b)=\emptyset$ by Theorem \ref{empty}.
\end{proof}

Our next goal is to prove that if $w_{k,l}\in \SAdm(\mu)$ satisfies $X_{w_{k,l}}(b)=\emptyset$, then $k<l$ satisfies one of the conditions in Lemma \ref{empty lemm}.
\begin{lemm}
\label{reduction 0}
Assume that $\frac{n+3}{2}\le l\le n-1$.
Then there exists $s\in \tS$ and $w'\in \tW$ such that $w_{3,l}\approx w'$, $sw'\sigma(s)\approx w_{1,l}$ and $sw'\approx w_{2,l}$.
\end{lemm}
\begin{proof}
Set $w'=s_1s_2\cdots s_{l-3}s_{n-1}s_{n-2}s_{n-1}\tau$.
Clearly, $\ell(w')=l$.
So the diagram
\begin{align*}
w_{3,l}=&s_0s_1\cdots s_{l-3}s_{n-1}s_0\tau\\
\xrightarrow{s_{n-2}s_0}&s_1s_2\cdots s_{l-3}s_{n-2}s_{n-1}s_{n-2}\tau\\
=& s_1s_2\cdots s_{l-3}s_{n-1}s_{n-2}s_{n-1}\tau=w'
\end{align*}
shows $w_{3,l}\approx w'$.
Set $s=s_{n-1}$.
Clearly, $\ell(sw'\sigma(s))=l-2$ and $\ell(sw')=l-1$.
Moreover, $sw'\sigma(s)\xrightarrow{s_0}s_0s_1\cdots s_{l-3}\tau=w_{1,l}$ and $sw'\xrightarrow{s_0s_{n-1}}s_0s_1\cdots s_{l-3}s_{n-1}\tau=w_{2,l}$.
This finishes the proof.
\end{proof}

\begin{lemm}
\label{reduction}
Assume that $4\le k\le \frac{n+1}{2}$ and $\frac{n+3}{2}\le l\le n-1$.
\begin{enumerate}[(i)]
\item If $k$ is odd and $k+l\le n+2$, then there exists $s\in \tS$ and $w'\in \tW$ such that $w_{k,l}\approx w'$, $sw'\sigma(s)\approx w_{k-2,l}$ and $sw'\approx w_{k-1,l}$.
\item If $l\equiv n-1$ and $k+l\geq n+4$, then there exists $s\in \tS$ and $w'\in \tW$ such that $w_{k,l}\approx w'$, $sw'\sigma(s)\approx w_{k,l-2}$ and $sw'\approx w_{k,l-1}$.
\item If $l\equiv n-1$ and $k+l=n+3$, then there exists $s\in \tS$ and $w'\in \tW$ such that $w_{k,l}\approx w'$, $sw'\sigma(s)\approx w_{k-1,l-1}$ and $sw'\approx w_{k,l-1}$.
\end{enumerate}
\end{lemm}
\begin{proof}
Note that $4\le k\le \frac{n+1}{2}$ implies $n\geq 7$ and hence $k\le n-3$.

Assume that (i) holds.
Then $l\le n-k+2\le n-2$.
Set 
\begin{align*}
w'&=s_{[l-3,1]}^{-1}s_{[k-5,1]}^{-1}s_{n-3}s_0s_1s_{n-1}s_0\tau\\
&=\tau s_{[l-1,3]}^{-1}s_{[k-3,3]}^{-1}s_{n-1}s_2s_3s_{1}s_2.
\end{align*}
It is easy to check that $s_{[l-1,3]}^{-1}s_{[k-3,3]}^{-1}s_{n-1}s_2s_3s_{1}s_2$ is a reduced expression.
Hence $\ell(w')=k+l-3$ and the diagram
\begin{align*}
w_{k,l}=&s_0s_1\cdots s_{l-3}s_{n-1}s_0\cdots s_{k-3}\tau\\
\xrightarrow{s_{n-2}s_{n-1}s_0}&s_{[l-3,1]}^{-1}s_{[k-3,0]}^{-1}s_{n-1}s_0\tau &(k<l\le n-2)\\
\xrightarrow{s_{[n-4,n-k+1]}}& s_{[n-4,n-k+1]}s_{[l-3,1]}^{-1}s_0s_1s_{n-1}s_0\tau\\
\xrightarrow{s_{[k-3,2]}^{-1}}&s_{[k-3,2]}^{-1}s_{[l-3,1]}^{-1}s_0s_1s_{n-1}s_0\tau &(l\le n-k+2)\\
=&s_{[l-3,1]}^{-1}s_{[k-4,1]}^{-1}s_0s_1s_{n-1}s_0\tau &(k<l)\\
\xrightarrow{s_{k-4}s_{n-k+2}}&s_{k-4}s_{[l-3,1]}^{-1}s_{[k-5,1]}^{-1}s_0s_1s_{n-1}s_0\tau &(l\le n-k+2\ \ \&\ \ 7\le k\le \tfrac{n+1}{2})\\
=&s_{[l-3,1]}^{-1}s_{[k-5,1]}^{-1}s_{k-6}s_0s_1s_{n-1}s_0\tau\\
\xrightarrow{s_{k-6}s_{n-k+4}}&s_{k-6}s_{[l-3,1]}^{-1}s_{[k-5,1]}^{-1}s_0s_1s_{n-1}s_0\tau &(l\le n-k+2\ \ \&\ \ 9\le k\le \tfrac{n+1}{2})\\
=&s_{[l-3,1]}^{-1}s_{[k-5,1]}^{-1}s_{k-8}s_0s_1s_{n-1}s_0\tau\\
\xrightarrow{\cdots}&\cdots\\
\xrightarrow{s_3s_{n-5}}&s_{[l-3,1]}^{-1}s_{[k-5,1]}^{-1}s_{1}s_0s_1s_{n-1}s_0\tau &(\text{$k$ is odd})\\
\xrightarrow{s_{n-3}s_{n-1}}&s_{[l-3,1]}^{-1}s_{[k-5,1]}^{-1}s_{n-3}s_0s_1s_{n-1}s_0\tau &(k<l\le n-2)\\
=&w'
\end{align*}
shows $w_{k,l}\approx w'$.
Here each condition in () is used to deduce $\rightarrow$ or $=$ on the same row, and $(s_3s_{n-5})\cdots (s_{k-6}s_{n-k+4})(s_{k-4}s_{n-k+2})=1$ if $k=5$.
Set $s=s_1$.
Then $l\le n-2$ implies
\begin{align*}
sw'\sigma(s)=s_{[l-3,2]}^{-1}s_{[k-5,1]}^{-1}s_0s_1s_{n-1}s_0\tau\xrightarrow{s_0s_1s_{[n-1,n-3]}s_{[n,n-2]}}s_{[l-3,0]}^{-1}s_{n-1}s_{[k-5,0]}^{-1}\tau=w_{k-2,l}.
\end{align*}
This shows $\ell(sw'\sigma(s))=k+l-5$ and $sw'\sigma(s)\approx w_{k-2,l}$.
Similarly, we have
\begin{align*}
sw'=s_{[l-3,2]}^{-1}s_{[k-5,1]}^{-1}s_{n-3}s_0s_1s_{n-1}s_0\tau\xrightarrow{s_0s_1s_{n-1}s_{n-2}s_0s_{n-3}s_{n-1}s_{n-2}}s_1s_{[l-3,0]}^{-1}s_{n-1}s_{[k-5,0]}^{-1}\tau
\end{align*}
and
$$s_1s_{[l-3,0]}^{-1}s_{n-1}s_{[k-5,0]}^{-1}\tau\xrightarrow{(s_{n-k+2}s_{k-4})\cdots (s_{n-5}s_3)(s_{n-3}s_1)}s_{[l-3,0]}^{-1}s_{n-1}s_{[k-4,0]}^{-1}\tau=w_{k-1,l}$$
by $l\le n-2$ and $k\le n-l+2$.
This shows $\ell(sw')=k+l-4$ and $sw'\approx w_{k-1,l}$.

Assume that $l\equiv n-1$ and $k+l\geq n+3$.
Then $l-k\geq n-2k+3\geq 2$ and $l\geq n-k+3\geq \frac{n+5}{2}$.
So we have
\begin{align*}
w_{k,l}=&s_0s_1\cdots s_{l-3}s_{n-1}s_0\cdots s_{k-3}\tau\\
\xrightarrow{s_{n-2}s_{n-1}s_0}&s_{[l-3,1]}^{-1}s_{[k-3,0]}^{-1}s_{n-1}s_0\tau &(k<l\le n-1)\\
\xrightarrow{s_{[n-4,n-k+1]}}& s_{[n-4,n-k+1]}s_{[l-3,1]}^{-1}s_0s_1s_{n-1}s_0\tau \\
\xrightarrow{s_{l-2}s_{n-l+1}}& s_{l-2}s_{[n-4,n-k+1]}s_{[l-4,1]}^{-1}s_0s_1s_{n-1}s_0\tau &(l-k\geq 2\ \ \&\ \ l\geq \tfrac{n+5}{2})\\
=&s_{[n-4,n-k+1]}s_{[l-4,1]}^{-1}s_{l-1}s_0s_1s_{n-1}s_0\tau &(n-k+3\le l\le n-3)\\
\xrightarrow{s_ls_{n-l-1}}& s_ls_{[n-4,n-k+1]}s_{[l-4,1]}^{-1}s_0s_1s_{n-1}s_0\tau &(l\geq \tfrac{n+5}{2})\\
=&s_{[n-4,n-k+1]}s_{[l-4,1]}^{-1}s_{l+1}s_0s_1s_{n-1}s_0\tau &(n-k+3\le l\le n-5)\\
\xrightarrow{\cdots}&\cdots\\
\xrightarrow{s_{n-5}s_4}& s_{[n-4,n-k+1]}s_{[l-4,1]}^{-1}s_{n-4}s_0s_1s_{n-1}s_0\tau &(l\equiv n-1)\\
\xrightarrow{s_2}&s_2 s_{[n-4,n-k+1]}s_{[l-4,1]}^{-1}s_0s_1s_{n-1}s_0\tau \\
=&s_{[n-4,n-k+1]}s_{[l-4,1]}^{-1}s_1s_0s_1s_{n-1}s_0\tau &(k\le n-3)\\
\xrightarrow{s_{n-3}s_{n-1}}& s_{[n-3,n-k+1]}s_{[l-4,1]}^{-1}s_0s_1s_{n-1}s_0\tau.
\end{align*}
Similarly as above, each condition in () is used to deduce $\rightarrow$ or $=$ on the same row, and $(s_{n-5}s_4)\cdots (s_ls_{n-l-1})(s_{l-2}s_{n-l+1})=1$ if $l=n-1$.

Assume that (ii) holds.
Then $k\geq n-l+4\geq 5$.
Set 
\begin{align*}
w'&=s_{[k-2,2]}^{-1}s_{[k-3,n-l+2]}s_{n-3}s_{[n-k-1,1]}^{-1}s_0s_1s_{n-1}s_0\tau\\
&=\tau s_{[k,4]}^{-1}s_{[k-1,n-l+4]}s_{n-1}s_{[n-k+1,3]}^{-1}s_2s_3s_1s_2\\
&=\tau s_{[k,4]}^{-1}s_{[k-1,n-l+4]}s_{n-1}s_{[k,3]}^{-1}s_{[n-k+1,k+1]}^{-1}s_2s_3s_1s_2\\
&=\tau s_{[k,4]}^{-1}s_{[k,3]}^{-1}s_{[k-2,n-l+3]}s_2s_3s_1s_2s_{n-1}s_{[n-k+1,k+1]}^{-1}.
\end{align*}
It is easy to check that $s_{[k,4]}^{-1}s_{[k,3]}^{-1}s_{[k-2,n-l+3]}s_2s_3s_1s_2s_{n-1}s_{[n-k+1,k+1]}^{-1}$ is a reduced expression.
Hence $\ell(w')=k+l-3$ and the diagram
\begin{align*}
w_{k,l}\rightarrow& s_{[n-3,n-k+1]}s_{[l-4,1]}^{-1}s_0s_1s_{n-1}s_0\tau \\
\xrightarrow{s_{[k-3,n-l+2]}}&s_{[k-3,n-l+2]}s_{[n-3,n-k+1]}s_{[n-k,1]}^{-1}s_0s_1s_{n-1}s_0\tau &(l\geq n-k+4)\\
\xrightarrow{s_{[k-2,2]}^{-1}}&s_{[k-2,2]}^{-1}s_{[k-3,n-l+2]}s_{n-3}s_{[n-k-1,1]}^{-1}s_0s_1s_{n-1}s_0\tau\\
=&w'
\end{align*}
shows $w_{k,l}\approx w'$.
Set $s=s_1$.
Then $5\le k\le \tfrac{n+1}{2}$ implies the diagram
\begin{align*}
sw'\sigma(s)=&s_{[k-2,1]}^{-1}s_{[k-3,n-l+2]}s_{[n-k-1,1]}^{-1}s_0s_1s_{n-1}s_0\tau\\
=&s_{[k-2,1]}^{-1}s_{[n-k-1,1]}^{-1}s_{[k-4,n-l+1]}s_0s_1s_{n-1}s_0\tau\\
=&s_{[k-2,2]}^{-1}s_{[k-3,1]}^{-1}s_{[n-k-1,k-1]}^{-1}s_{[k-4,n-l+1]}s_0s_1s_{n-1}s_0\tau\\
=&s_{[n-k-1,2]}^{-1}s_{[k-3,1]}^{-1}s_{[k-4,n-l+1]}s_0s_1s_{n-1}s_0\tau\\
\xrightarrow{s_{[n,n-2]}}& s_0s_{n-1}s_{[n-k-1,2]}^{-1}s_{[k-3,1]}^{-1}s_{[k-4,n-l+1]}s_0s_1\tau\\
\xrightarrow{s_{[n-1,n-3]}}& s_{[n-1,n-3]}s_0s_{[n-k-1,2]}^{-1}s_{[k-3,1]}^{-1}s_{[k-4,n-l+1]}\tau \\
=&s_{[n-1,n-3]}s_{[n-k-1,2]}^{-1}s_{[k-3,0]}^{-1}s_{[k-4,n-l+1]}\tau\\
\xrightarrow{s_0s_1}& s_0s_1s_{n-1}s_{[n-k-1,2]}^{-1}s_{[k-3,0]}^{-1}s_{[k-4,n-l+1]}\tau \\
=& s_{[k-2,n-l+3]}s_{[n-k-1,0]}^{-1}s_{n-1}s_{[k-3,0]}^{-1}\tau\\
\xrightarrow{s_{[k-2,n-l+3]}^{-1}}&s_{[n-k-1,0]}^{-1}s_{n-1}s_{[k-3,0]}^{-1}s_{[l-5,n-k]}^{-1}\tau \\
=&s_{[l-5,0]}^{-1}s_{n-1}s_{[k-3,0]}^{-1}\tau\\
=&w_{k,l-2},
\end{align*}
which shows $\ell(sw'\sigma(s))=k+l-5$ and $sw'\sigma(s)\approx w_{k,l-2}$.
Similarly, we have
\begin{align*}
sw'=&s_{[k-2,1]}^{-1}s_{[k-3,n-l+2]}s_{n-3}s_{[n-k-1,1]}^{-1}s_0s_1s_{n-1}s_0\tau\\
=&s_{[n-k-1,2]}^{-1}s_{[k-3,1]}^{-1}s_{[k-4,n-l+1]}s_{n-3}s_0s_1s_{n-1}s_0\tau\\ 
\xrightarrow{s_1}&s_{[n-k-1,1]}^{-1}s_{[k-3,1]}^{-1}s_{[k-4,n-l+1]}s_0s_1s_{n-1}s_0\tau\\
\xrightarrow{s_{[n,n-2]}}& s_0s_{n-1}s_{[n-k-1,1]}^{-1}s_{[k-3,1]}^{-1}s_{[k-4,n-l+1]}s_0s_1\tau\\
\xrightarrow{s_{[n-1,n-3]}}& s_{[n-1,n-3]}s_{[n-k-1,0]}^{-1}s_{[k-3,1]}^{-1}s_{[k-4,n-l+1]}\tau\\
\xrightarrow{s_0s_1}& s_0s_1s_{n-1}s_{[n-k-1,0]}^{-1}s_{[k-3,1]}^{-1}s_{[k-4,n-l+1]}\tau\\
=& s_{[k-2,n-l+3]}s_{n-1}s_{[n-k-1,0]}^{-1}s_{n-1}s_{[k-3,0]}^{-1}\tau \\
\xrightarrow{s_{[k-2,n-l+3]}^{-1}}&s_{n-1}s_{[n-k-1,0]}^{-1}s_{n-1}s_{[k-3,0]}^{-1}s_{[l-5,n-k]}^{-1}\tau\\
=&s_{n-1}s_{[l-5,0]}^{-1}s_{n-1}s_{[k-3,0]}^{-1}\tau.
\end{align*}
Since $l\equiv n-1$ and $l-k\geq n-2k+4\geq 3$, 
$$s_{n-1}s_{[l-5,0]}^{-1}s_{n-1}s_{[k-3,0]}^{-1}\tau\xrightarrow{(s_{n-l+2}s_{l-2})\cdots(s_3s_{n-3})(s_1s_{n-1})}s_{[l-4,0]}^{-1}s_{n-1}s_{[k-3,0]}^{-1}\tau=w_{k,l-1}.$$
This shows $\ell(sw')=k+l-4$ and $sw'\approx w_{k,l-1}$.

Assume that (iii) holds.
Set
\begin{align*}
w'&=s_{[k-3,2]}^{-1}s_{n-3}s_{[n-k-1,1]}^{-1}s_0s_1s_{n-1}s_0\tau\\
&=\tau s_{[k-1,4]}^{-1}s_{n-1}s_{[n-k+1,3]}^{-1}s_2s_3s_1s_2
\end{align*}
It is easy to check that $ s_{[k-1,4]}^{-1}s_{n-1}s_{[n-k+1,3]}^{-1}s_2s_3s_1s_2$ is a reduced expression.
Hence $\ell(w')=k+l-3$ and the diagram
\begin{align*}
w_{k,l}\rightarrow& s_{[n-3,n-k+1]}s_{[n-k-1,1]}^{-1}s_0s_1s_{n-1}s_0\tau \\
\xrightarrow{s_{[k-3,2]}^{-1}}&s_{[k-3,2]}^{-1}s_{n-3}s_{[n-k-1,1]}^{-1}s_0s_1s_{n-1}s_0\tau\\
=&w'
\end{align*}
shows $w_{k,l}\approx w'$.
Set $s=s_1$.
Then $4\le k\le \frac{n+1}{2}$ implies the diagram
\begin{align*}
sw'\sigma(s)=&s_{[k-3,1]}^{-1}s_{[n-k-1,1]}^{-1}s_0s_1s_{n-1}s_0\tau\\
=&s_{[k-3,2]}^{-1}s_{[k-4,1]}^{-1}s_{[n-k-1,k-2]}^{-1}s_0s_1s_{n-1}s_0\tau\\
=&s_{[n-k-1,2]}^{-1}s_{[k-4,1]}^{-1}s_0s_1s_{n-1}s_0\tau\\
\xrightarrow{s_{[n,n-2]}}&s_0s_{n-1}s_{[n-k-1,2]}^{-1}s_{[k-4,1]}^{-1}s_0s_1\tau\\
\xrightarrow{s_{[n-1,n-3]}}&s_{[n-1,n-3]}s_0s_{[n-k-1,2]}^{-1}s_{[k-4,1]}^{-1}\tau\\
=&s_{[n-1,n-3]}s_{[n-k-1,2]}^{-1}s_{[k-4,0]}^{-1}\tau\\
\xrightarrow{s_0s_1}& s_0s_1s_{n-1}s_{[n-k-1,2]}^{-1}s_{[k-4,0]}^{-1}\tau\\
=&s_{[n-k-1,0]}^{-1}s_{n-1}s_{[k-4,0]}^{-1}\tau\\
=&w_{k-1,l-1},
\end{align*}
which shows $\ell(sw'\sigma(s))=k+l-5$ and $sw'\sigma(s)\approx w_{k-1,l-1}$.
Similarly, we have
\begin{align*}
sw'=&s_{[k-3,1]}^{-1}s_{n-3}s_{[n-k-1,1]}^{-1}s_0s_1s_{n-1}s_0\tau\\
=&s_{[n-k-1,2]}^{-1}s_{[k-4,1]}^{-1}s_{n-3}s_0s_1s_{n-1}s_0\tau\\
\xrightarrow{s_1}&s_{[n-k-1,1]}^{-1}s_{[k-4,1]}^{-1}s_0s_1s_{n-1}s_0\tau\\
\xrightarrow{s_{[n,n-2]}}&s_0s_{n-1}s_{[n-k-1,1]}^{-1}s_{[k-4,1]}^{-1}s_0s_1\tau\\
\xrightarrow{s_{[n-1,n-3]}}&s_{[n-1,n-3]}s_{[n-k-1,0]}^{-1}s_{[k-4,1]}^{-1}\tau\\
\xrightarrow{s_0s_1}& s_0s_1s_{n-1}s_{[n-k-1,0]}^{-1}s_{[k-4,1]}^{-1}\tau\\
=&s_{n-1}s_{[n-k-1,0]}^{-1}s_{n-1}s_{[k-4,0]}^{-1}\tau.
\end{align*}
Since $l\equiv n-1$ and hence $k$ is even, 
\begin{align*}
s_{n-1}s_{[n-k-1,0]}^{-1}s_{n-1}s_{[k-4,0]}^{-1}\tau\xrightarrow{s_{n-k+1}(s_{k-3}s_{n-k+3})\cdots(s_3s_{n-3})(s_1s_{n-1})}&s_{[n-k-1,0]}^{-1}s_{n-1}s_{[k-3,0]}^{-1}\tau\\
=&w_{k,l-1}.
\end{align*}
This shows $\ell(sw')=k+l-4$ and $sw'\approx w_{k,l-1}$.
\end{proof}

For $w_{k,l}\in \SAdm(\mu)_{\neq \DL}$, let $w'_{k,l}$ denote $w_{k-2,l}$ (resp.\ $w_{k,l-2}$, resp.\ $w_{k-1,l-1}$) if $k+l\le n+2$ (resp.\ $k+l\geq n+4$, resp.\ $k+l=n+3$).
Combining the above three lemmas with Proposition \ref{DL method prop} and Theorem \ref{empty}, we obtain the following theorem.
\begin{theo}
\label{non-empty}
We have
\begin{align*}
\SAdm(\mu)_{\DL}=\{w_{k,l}\mid \text{$k=1$ or $l\le \tfrac{n+2}{2}$}\}.
\end{align*}
Moreover, $w_{k,l}\in\SAdm(\mu)_{\neq\DL}$ if and only if $3\le k< \frac{n+2}{2}<l\le n-1$ and one of the following conditions is satisfied:
\begin{enumerate}[(i)]
\item $k$ is odd and $k+l\le n+2$.
\item $l\equiv n-1$ and $k+l\geq n+3$.
\end{enumerate}
If this is the case, $X_{w_{k,l}}(\tau)$ is $\J_\tau$-equivariant universally homeomorphic to a Zariski-locally trivial $\A^1$-bundle over $X_{w'_{k,l}}(\tau)$.
In particular, if $w_{k,l}\in\SAdm(\mu)_{\neq\DL}$ satisfies (i) (resp.\ (ii)), then $X_{w_{k,l}}(\tau)$ is an iterated fibration of rank $\frac{k-1}{2}$ (resp.\ $k+\frac{l-n-3}{2}$) over $X_{w_{1,l}}(\tau)$ (resp.\ $X_{w_{1,n-k+2}}(\tau)$).
\end{theo}

\begin{exam}
\label{13 14}
In the case $n=13$, the non-empty Ekedahl-Oort strata are the following.
\begin{center}
\begin{tikzpicture}
\draw (-1,9)node {$w_{1,13}$};
\draw (0,12)node{$w_{7,12}$};
\draw (0,11.5)node{$w_{6,12}$};
\draw (0,11)node{$w_{5,12}$};
\draw (0,10.5)node {$w_{4,12}$};
\draw (0,10)node {$w_{3,12}$};
\draw (0,9)node {$w_{1,12}$};
\draw (1,10)node {$w_{3,11}$};
\draw (1,9)node {$w_{1,11}$};
\draw (2,12)node{$w_{7,10}$};
\draw (2,11.5)node{$w_{6,10}$};
\draw (2,11)node{$w_{5,10}$};
\draw (2,10)node {$w_{3,10}$};
\draw (2,9)node {$w_{1,10}$};
\draw (3,11)node{$w_{5,9}$};
\draw (3,10)node {$w_{3,9}$};
\draw (3,9)node {$w_{1,9}$};
\draw (4,12)node{$w_{7,8}$};
\draw (4,11)node{$w_{5,8}$};
\draw (4,10)node {$w_{3,8}$};
\draw (4,9)node {$w_{1,8}$};
\draw (5,11.5)node{$w_{6,7}$};
\draw (5,11)node{$w_{5,7}$};
\draw (5,10.5)node {$w_{4,7}$};
\draw (5,10)node {$w_{3,7}$};
\draw (5,9.5)node {$w_{2,7}$};
\draw (5,9)node {$w_{1,7}$};
\draw (6,11)node{$w_{5,6}$};
\draw (6,10.5)node {$w_{4,6}$};
\draw (6,10)node {$w_{3,6}$};
\draw (6,9.5)node {$w_{2,6}$};
\draw (6,9)node {$w_{1,6}$};
\draw (7,10.5)node {$w_{4,5}$};
\draw (7,10)node {$w_{3,5}$};
\draw (7,9.5)node {$w_{2,5}$};
\draw (7,9)node {$w_{1,5}$};
\draw (8,10)node {$w_{3,4}$};
\draw (8,9.5)node {$w_{2,4}$};
\draw (8,9)node {$w_{1,4}$};
\draw (9,9.5)node {$w_{2,3}$};
\draw (9,9)node {$w_{1,3}$};
\draw (10,9)node {$w_{1,2}$};
\draw [thick, ->] (0.5,12.04)--(1.5,12.04);
\draw [thick, ->] (2.5,12.04)--(3.5,12.04);
\draw [thick, ->] (0.5,11.54)--(1.5,11.54);
\draw [thick, ->] (0.5,11.04)--(1.5,11.04);
\draw [thick, ->] (0,9.8)--(0,9.2);
\draw [thick, ->] (1,9.8)--(1,9.2);
\draw [thick, ->] (2,9.8)--(2,9.2);
\draw [thick, ->] (3,9.8)--(3,9.2);
\draw [thick, ->] (4,9.8)--(4,9.2);
\draw [thick, ->] (2,10.8)--(2,10.2);
\draw [thick, ->] (3,10.8)--(3,10.2);
\draw [thick, ->] (4,10.8)--(4,10.2);
\draw [thick, ->] (4,11.8)--(4,11.2);
\draw [thick, ->] (0.5,10.5)--(1,10.15);
\draw [thick, ->] (2.5,11.5)--(3,11.15);
\end{tikzpicture}
\end{center}
In the case $n=14$, the non-empty Ekedahl-Oort strata are the following.
\begin{center}
\begin{tikzpicture}
\draw (-1,9)node {$w_{1,14}$};
\draw (0,12)node{$w_{7,13}$};
\draw (0,11.5)node{$w_{6,13}$};
\draw (0,11)node{$w_{5,13}$};
\draw (0,10.5)node {$w_{4,13}$};
\draw (0,10)node {$w_{3,13}$};
\draw (0,9)node {$w_{1,13}$};
\draw (1,10)node {$w_{3,12}$};
\draw (1,9)node {$w_{1,12}$};
\draw (2,12)node{$w_{7,11}$};
\draw (2,11.5)node{$w_{6,11}$};
\draw (2,11)node{$w_{5,11}$};
\draw (2,10)node {$w_{3,11}$};
\draw (2,9)node {$w_{1,11}$};
\draw (3,11)node{$w_{5,10}$};
\draw (3,10)node {$w_{3,10}$};
\draw (3,9)node {$w_{1,10}$};
\draw (4,12)node{$w_{7,9}$};
\draw (4,11)node{$w_{5,9}$};
\draw (4,10)node {$w_{3,9}$};
\draw (4,9)node {$w_{1,9}$};
\draw (5,12)node{$w_{7,8}$};
\draw (5,11.5)node{$w_{6,8}$};
\draw (5,11)node{$w_{5,8}$};
\draw (5,10.5)node {$w_{4,8}$};
\draw (5,10)node {$w_{3,8}$};
\draw (5,9.5)node {$w_{2,8}$};
\draw (5,9)node {$w_{1,8}$};
\draw (6,11.5)node{$w_{6,7}$};
\draw (6,11)node{$w_{5,7}$};
\draw (6,10.5)node {$w_{4,7}$};
\draw (6,10)node {$w_{3,7}$};
\draw (6,9.5)node {$w_{2,7}$};
\draw (6,9)node {$w_{1,7}$};
\draw (7,11)node {$w_{5,6}$};
\draw (7,10.5)node {$w_{4,6}$};
\draw (7,10)node {$w_{3,6}$};
\draw (7,9.5)node {$w_{2,6}$};
\draw (7,9)node {$w_{1,6}$};
\draw (8,10.5)node {$w_{4,5}$};
\draw (8,10)node {$w_{3,5}$};
\draw (8,9.5)node {$w_{2,5}$};
\draw (8,9)node {$w_{1,5}$};
\draw (9,10)node {$w_{3,4}$};
\draw (9,9.5)node {$w_{2,4}$};
\draw (9,9)node {$w_{1,4}$};
\draw (10,9.5)node {$w_{2,3}$};
\draw (10,9)node {$w_{1,3}$};
\draw (11,9)node {$w_{1,2}$};
\draw [thick, ->] (0.5,12.04)--(1.5,12.04);
\draw [thick, ->] (2.5,12.04)--(3.5,12.04);
\draw [thick, ->] (0.5,11.54)--(1.5,11.54);
\draw [thick, ->] (0.5,11.04)--(1.5,11.04);
\draw [thick, ->] (0,9.8)--(0,9.2);
\draw [thick, ->] (1,9.8)--(1,9.2);
\draw [thick, ->] (2,9.8)--(2,9.2);
\draw [thick, ->] (3,9.8)--(3,9.2);
\draw [thick, ->] (4,9.8)--(4,9.2);
\draw [thick, ->] (2,10.8)--(2,10.2);
\draw [thick, ->] (3,10.8)--(3,10.2);
\draw [thick, ->] (4,10.8)--(4,10.2);
\draw [thick, ->] (4,11.8)--(4,11.2);
\draw [thick, ->] (0.5,10.5)--(1,10.15);
\draw [thick, ->] (2.5,11.5)--(3,11.15);
\end{tikzpicture}
\end{center}
\end{exam}

\begin{rema}
\label{positive Coxeter}
Assume that $n$ is odd.
If $k=\frac{n+1}{2}$ or $l=\frac{n+3}{2}$, then $w_{k,l}\in \SAdm(\mu)_{\neq \DL}$ is of positive Coxeter type.
Indeed, $w_{\frac{n+1}{2},n-1}$ is of positive Coxeter type because its finite part is a $\sigma$-Coxeter element.
By Lemma \ref{reduction} and \cite[Lemma 5.1 \& Lemma 5.2]{SSY23}, the elements satisfying $k=\frac{n+1}{2}$ or $l=\frac{n+3}{2}$ are also of positive Coxeter type.
On the other hand, the elements $w_{k,l}\in \SAdm(\mu)_{\neq \DL}$ with $k<\frac{n+1}{2}$ and $l>\frac{n+3}{2}$ are not of positive Coxeter type because $w_{1,l}$ is not a $\sigma$-Coxeter element if $l>\frac{n+3}{2}$ (cf.\ \cite[\S 5.3]{SSY23}).

Assume that $n$ is even.
Similarly as the odd case, $w_{k,l}\in \SAdm(\mu)_{\neq \DL}$ is of positive Coxeter type if and only if $k=\frac{n}{2}$ or $l=\frac{n}{2}+2$.
\end{rema}

\section{Geometric Structure of Non-Empty Ekedahl-Oort Strata}
\label{geometric structure}
Keep the notations and assumptions in \S\ref{non-empty section}.
\subsection{Iterated Fibrations over Deligne-Lusztig Varieties}
For $w_{k,l}\in \SAdm(\mu)_{\DL}$, we have
\begin{align*}
\supp_\sigma(w_{k,l})=
\begin{cases}
\{s_i,s_{n-i-2}\mid 0\le i\le l-3\}\sqcup \{s_{n-1}\} & (k\geq 2)\\
\{s_i,s_{n-i-2}\mid 0\le i\le l-3\} & (k=1\ \&\ l\le \frac{n+2}{2})\\
\tS\setminus\{s_{n-1}\} & (k=1\ \&\ l\geq \frac{n+3}{2}).
\end{cases}
\end{align*}
By Proposition \ref{spherical}, the following lemma completes the description of $\pi(X_{w_{k,l}}(\tau))$ in this case.
\begin{lemm}
\label{SwDL}
Let $w_{k,l}\in \SAdm(\mu)_{\DL}$.
If $l=k+1$, then
\begin{align*}
S(w_{k,k+1},\sigma)=
\begin{cases}
\{s_i\mid k\le i\le n-k-2\}\sqcup\\
\{s_1,s_3,\ldots, s_{k-2}, s_{n-1},s_{n-3},\ldots, s_{n-k}\} & (\text{odd $k$})\\
\{s_i\mid k\le i\le n-k-2\}& (\text{even $k$}).
\end{cases}
\end{align*}
If $l-k\geq 2$, then
\begin{align*}
S(w_{k,l},\sigma)=
\begin{cases}
\{s_i\mid l-1\le i\le n-l-1\}& (\text{$l\le\frac{n+2}{2}$})\\
\{s_i\mid n-l+2\le i\le l-3\}& (\text{$k=1$ \& $l\geq \frac{n+3}{2}$}).
\end{cases}
\end{align*}
In particular, 
\begin{align*}
\supp_\sigma(w_{k,l})\cup S(w_{k,l},\sigma)=
\begin{cases}
\tS\setminus\{s_{l-2},s_{n-l}\} & (k\geq 2\ \text{or} \ l=2)\\
\tS\setminus\{s_{l-2},s_{n-l},s_{n-1}\} & (k=1\ \&\ 3\le l\le \frac{n+2}{2})\\
\tS\setminus\{s_{n-1}\} & (k=1\ \&\ l\geq \frac{n+3}{2}).
\end{cases}
\end{align*}
\end{lemm}
\begin{proof}
Set $y_{k,l}=s_{[n-2,k]}s_{[n-1,l]}(=p(w_{k,l}))$.
Assume that $l=k+1$.
Then 
\begin{align*}
y_{k,k+1}\sigma(s_i)y_{k,k+1}^{-1}=\begin{cases}
s_{n-i-2} & (1\le i\le n-k-2)\\
(k\ n) & (i=n-k-1)\\
s_{n-1} & (i=n-k)\\
(k-1\ n-1) & (i=n-k+1)\\
s_{n-i} & (n-k+2\le i\le n-1).
\end{cases}
\end{align*}
By Theorem \ref{non-empty}, we have $k\le \frac{n}{2}$, i.e., $k-2\le n-k-2$.
Note that $\vp^\mu s_i\vp^{-\mu}=s_i$ unless $i=0,n-2$.
So we have
\begin{align*}
\Ad(w_{k,k+1})\sigma(s_i)=
\begin{cases}
s_{n-i-2} & (1\le i\le k-2)\\
s_{n-i} & (n-k+2\le i\le n-1).
\end{cases}
\end{align*}
It follows from $\Ad(w_{k,k+1})\sigma(s_{n-k+1})\notin S$ that if $k$ is odd (resp.\ even), then
\begin{align*}
&\{s_1,s_3.\ldots, s_{k-2}, s_{n-1},s_{n-3},\ldots, s_{n-k}\}\subset S(w_{k,k+1},\sigma)\quad \text{and}\\
&\{s_2,s_4,\ldots,s_{k-3},s_{n-2},s_{n-4},\ldots, s_{n-k+1}\}\cap S(w_{k,k+1},\sigma)=\emptyset\\
(&\text{resp.\ }\{s_1,s_2,\ldots,s_{k-2},s_{n-1},s_{n-2},\ldots, s_{n-k}\}\cap S(w_{k,k+1},\sigma)=\emptyset).
\end{align*} 
If $k\le \frac{n-1}{2}$ and hence $k-1<\frac{n-2}{2}<n-k-1$, then 
\begin{align*}
\Ad(w_{k,k+1})\sigma(s_i)=s_{n-i-2}
\end{align*}
for $k-1\le i\le \frac{n-2}{2}$.
So it follows from $\Ad(w_{k,k+1})\sigma(s_{n-k-1})\notin S$ that
\begin{align*}
s_{k-1},s_{n-k-1}\notin S(w_{k,k+1},\sigma)\quad\text{and}\quad s_i\in S(w_{k,k+1},\sigma) 
\end{align*}
for $k\le i\le n-k-2$.
If $k=\frac{n}{2}$, then $s_{k-1}=s_{n-k-1}\notin S(w_{k,k+1},\sigma)$.

Assume that $l-k\geq 2$.
Then
\begin{align*}
y_{k,l}\sigma(s_i)y_{k,l}^{-1}=\begin{cases}
s_{n-i-2} & (1\le i\le n-l-1)\\
(l-1\ n) & (i=n-l)\\
(l-2\ n) & (i=n-l+1)\\
s_{n-i-1} & (n-l+2\le i\le n-k-1)\\
(k\ n-1) & (i=n-k)\\
(k-1\ n-1) & (i=n-k+1)\\
s_{n-i} & (n-k+2\le i\le n-1).
\end{cases}
\end{align*}
By Theorem \ref{non-empty}, we have $k+l\le n+1$, i.e., $k-2\le n-l-1$.
So we have
\begin{align*}
\Ad(w_{k,l})\sigma(s_i)=
\begin{cases}
s_{n-i-2} & (1\le i\le k-2)\\
s_{n-i} & (n-k+2\le i\le n-1).
\end{cases}
\end{align*}
It follows from $\Ad(w_{k,l})\sigma(s_{n-k}),\Ad(w_{k,l})\sigma(s_{n-k+1})\notin S$ that
\begin{align*}
\{s_1,s_2,\ldots,s_{k-2},s_{n-1},s_{n-2},\ldots, s_{n-k}\}\cap S(w_{k,l},\sigma)=\emptyset.
\end{align*} 
If $l\le \frac{n+2}{2}$, i.e., $l-3\le n-l-1$, then
\begin{align*}
\Ad(w_{k,l})\sigma(s_i)=
\begin{cases}
s_{n-i-2} & (k-1\le i\le l-3)\\
s_{n-i-1} & (n-l+2\le i\le n-k-1).
\end{cases}
\end{align*}
So it follows from $\Ad(w_{k,l})\sigma(s_{n-l+1})\notin S$ that
\begin{align*}
\{s_{k-1},s_{k},\ldots,s_{l-3},s_{n-k-1},s_{n-k-2},\ldots, s_{n-l+1}\}\cap S(w_{k,l},\sigma)=\emptyset.
\end{align*}
If $l\le \frac{n+1}{2}$, i.e., $l-2\le n-l-1$, then
\begin{align*}
\Ad(w_{k,l})\sigma(s_i)=s_{n-i-2}
\end{align*}
for $l-2\le i\le n-l-1$.
So it follows from $\Ad(w_{k,l})\sigma(s_{n-l})\notin S$ that
\begin{align*}
s_{l-2},s_{n-l}\notin S(w_{k,l},\sigma)\quad \text{and}\quad s_i\in S(w_{k,l},\sigma)
\end{align*}
for $l-1\le i\le n-l-1$.
If $l=\frac{n+2}{2}$, then $s_{l-2}=s_{n-l}\notin S(w_{k,l},\sigma)$.

Assume that $k=1$ and $l\geq \frac{n+3}{2}(>2)$.
Then $\Ad(w_{1,l})\sigma(s_{n-1})\notin S$ and hence $s_{n-1}\notin S(w_{1,l},\sigma)$.
By  $l\geq \frac{n+3}{2}$, i.e., $n-l+2\le l-1$, we have
\begin{align*}
\Ad(w_{1,l})\sigma(s_i)=
\begin{cases}
s_{n-i-2} & (1\le i\le n-l-1)\\
s_{n-i-1} & (l-1\le i\le n-2).
\end{cases}
\end{align*}
It follows from $\Ad(w_{1,l})\sigma(s_{n-l})\notin S$ that
\begin{align*}
\{s_1,s_2,\ldots, s_{n-l},s_{n-2},s_{n-3},\ldots, s_{l-1}\}\cap S(w_{1,l},\sigma)=\emptyset.
\end{align*}
If $l\geq \frac{n+4}{2}$, i.e., $n-l+2\le l-2$, then $\Ad(w_{1,l})\sigma(s_{l-2})=s_{n-l+1}$.
If $l=\frac{n+3}{2}$, then $s_{l-2}=s_{n-l+1}$.
It follows from $\Ad(w_{1,l})\sigma(s_{n-l+1})\notin S$ that $s_{l-2},s_{n-l+1}\notin S(w_{1,l},\sigma)$ in both cases.
Finally, $\Ad(w_{1,l})\sigma(s_i)=s_{n-i-1}$ and hence $s_i\in S(w_{1,l},\sigma)$ for $n-l+2\le i\le l-3$.

The statement is summary of these results.
\end{proof}

We denote by $J(w_{k,l})$ the set of all simple affine reflections appearing in the sequence defining $w_{k,l}\rightarrow w'_{k,l}$ in the proof of Lemma \ref{reduction 0} and Lemma \ref{reduction}.
Specifically,
\begin{align*}
J(w_{k,l})=
\begin{cases}
\{s_i,s_{n-i-2}\mid 0\le i\le k-3\}\sqcup \{s_{n-1}\} & (k+l\le n+2)\\
\{s_i,s_{n-i-2}\mid 0\le i\le k-3\}\sqcup \{s_{n-1},s_{k-2}\} & (k+l\geq n+3).
\end{cases}
\end{align*}

\begin{lemm}
\label{SwnDL}
Let $w_{k,l}\in \SAdm(\mu)_{\neq \DL}$.
Then 
\begin{align*}
S(w_{k,l},\sigma)=
\begin{cases}
S(w_{1,l},\sigma)=\{s_i\mid n-l+2\le i\le l-3\} & (k+l\le n+2)\\
S(w_{1,n-k+2},\sigma)=\{s_i\mid k\le i\le n-k-1\} & (k+l\geq n+3).
\end{cases}
\end{align*}
In particular, $J(w_{k,l})$ commutes with $S(w_{k,l},\sigma)=S(w'_{k,l},\sigma)$.
\end{lemm}
\begin{proof}
By Theorem \ref{non-empty}, we have $l\geq \frac{n+3}{2}$ and $n-k+2\geq \frac{n+3}{2}$.
So $S(w_{1,l},\sigma)=\{s_i\mid n-l+2\le i\le l-3\}$ and $S(w_{1,n-k+2},\sigma)=\{s_i\mid k\le i\le n-k-1\}$ by Lemma \ref{SwDL}.
If $k+l\le n+2$, then $k\le n-l+2$ and $l-3\le n-k-1$.
So if $k+l\le n+2$ (resp.\ $k+l\geq n+3$), $S(w_{1,l},\sigma)$ (resp.\ $S(w_{1,n-k+2},\sigma)$) and $J(w_{k,l})$ commute.
In particular, $S(w_{1,l},\sigma)\subseteq S(w_{k,l},\sigma)$ (resp.\ $S(w_{1,n-k+2},\sigma)\subseteq S(w_{k,l},\sigma)$).
It remains to show that this inclusion is actually equality.

Assume that $k+l\le n+1$.
Then $l-k\geq 2$ by $l\geq \frac{n+3}{2}$.
Similarly as the proof of Lemma \ref{SwDL}, we have
\begin{align*}
\{s_1,s_2,\ldots,s_{k-2},s_{n-1},s_{n-2},\ldots, s_{n-k}\}\cap S(w_{k,l},\sigma)=\emptyset
\end{align*} 
By $l\geq \frac{n+3}{2}$, i.e., $l-1\geq n-l+2$, then
\begin{align*}
\Ad(w_{k,l})\sigma(s_i)=
\begin{cases}
s_{n-i-2} & (k-1\le i\le n-l-1)\\
s_{n-i-1} & (l-1\le i\le n-k-1).
\end{cases}
\end{align*}
So it follows from $\Ad(w_{k,l})\sigma(s_{n-l})\notin S$ that
\begin{align*}
\{s_{k-1},s_{k},\ldots,s_{n-l},s_{n-k-1},s_{n-k-2},\ldots, s_{l-1}\}\cap S(w_{k,l},\sigma)=\emptyset.
\end{align*}
It is straightforward to check that $s_{l-2},s_{n-l+1}\notin S(w_{k,l},\sigma)$.
Combining these results with the above inequality, we have $S(w_{k,l},\sigma)=S(w_{1,l},\sigma)$.

Assume that $k+l=n+2$.
If $l=k+1$, then we must have $k=\frac{n+1}{2}$.
It is straightforward to check the equality $S(w_{k,l},\sigma)=S(w_{1,l},\sigma)=\emptyset$ in this case.
So we may assume that $l-k\geq 2$.
Note that $s_{k-2}=s_{n-l}$ and $s_{k-1}=s_{n-l+1}$.
We have
\begin{align*}
\Ad(w_{k,l})\sigma(s_i)=
\begin{cases}
s_{n-i-2} & (1\le i\le k-3=n-l-1)\\
s_{n-i} & (n-k+2\le i\le n-1).
\end{cases}
\end{align*}
Thus $\Ad(w_{k,l})\sigma(s_{n-l}),\Ad(w_{k,l})\sigma(s_{n-l+1}),\Ad(w_{k,l})\sigma(s_{n-k}),\Ad(w_{k,l})\sigma(s_{n-k+1})\notin S$ implies that
\begin{align*}
\{s_1,s_2,\ldots,s_{k-1},s_{n-1},s_{n-2},\ldots, s_{n-k}\}\cap S(w_{k,l},\sigma)=\emptyset.
\end{align*} 
Combining this with the above inequality, we have $S(w_{k,l},\sigma)=S(w_{1,l},\sigma)$.

Assume that $k+l\geq n+3$.
Then $l-k\geq 2$ by $k\le \frac{n+1}{2}$.
We have $n-k+2\le l-1$ and
\begin{align*}
\Ad(w_{k,l})\sigma(s_i)=
\begin{cases}
s_{n-i-2} & (1\le i\le n-l-1)\\
s_{n-i} & (l-1\le i\le n-1).
\end{cases}
\end{align*}
It follows from $\Ad(w_{k,l})\sigma(s_{n-l}),\Ad(w_{k,l})\sigma(s_{n-l+1})\notin S$ that
\begin{align*}
\{s_1,s_2,\ldots,s_{n-l+1},s_{n-1},s_{n-2},\ldots, s_{l-1}\}\cap S(w_{k,l},\sigma)=\emptyset.
\end{align*}
If $k\le \frac{n+1}{2}$, i.e., $k-2\le n-k-1$, then
\begin{align*}
\Ad(w_{k,l})\sigma(s_i)=
\begin{cases}
s_{n-i-1} & (n-l+2\le i\le k-2)\\
s_{n-i} & (n-k+2\le i\le l-2).
\end{cases}
\end{align*}
It follows from $\Ad(w_{k,l})\sigma(s_{n-k}),\Ad(w_{k,l})\sigma(s_{n-k+1})\notin S$ that
\begin{align*}
\{s_{n-l+2},s_{n-l+3},\ldots,s_{k-2},s_{l-2},s_{l-3},\ldots, s_{n-k}\}\cap S(w_{k,l},\sigma)=\emptyset.
\end{align*}
If $k\le\frac{n}{2}$, i.e., $k-1\le n-k-1$, then we also have $\Ad(w_{k,l})\sigma(s_{k-1})=s_{n-k}$.
If $k=\frac{n+1}{2}$, then $s_{k-1}=s_{n-k}$.
Thus $s_{k-1}\notin S(w_{k,l},\sigma)$ and hence $S(w_{k,l},\sigma)=S(w_{1,n-k+2},\sigma)$.
This finishes the proof.
\end{proof}

Now we can deduce the geometric structure of $X_{\mu}(b)$ via the Ekedahl-Oort stratification.
We set $\supp_\sigma(w)_1\coloneqq \tau_1\supp_\sigma(w)\tau_1^{-1}=\{s_{i+1}\mid s_i\in \supp_\sigma(w)\}$ and $S(w,\sigma)_1\coloneqq\tau_1 S(w,\sigma)\tau_1^{-1}=\{s_{i+1}\mid s_i\in S(w,\sigma)\}$.
Recall that for $w_{k,l}\in \SAdm(\mu)_{\DL}$, $Y(w_{k,l})$ is a Deligne-Lusztig variety in $P_{\supp_\sigma(w_{k,l})}/I$ as in Proposition \ref{spherical}.
Its image $\pi(Y(w_{k,l}))$ is isomorphic to a Deligne-Lusztig variety $\pi_{S(w_{k,l},\sigma)}(Y(w_{k,l}))$  in $P_{\supp_\sigma(w_{k,l})\cup S(w_{k,l},\sigma)}/P_{S(w_{k,l},\sigma)}$.
If $w_{k,l}\in\SAdm(\mu)_{\neq\DL}$ and $k+l\le n+2$ (resp.\ $k+l\geq n+3$), let $Y(w_{k,l})\subset X_{w_{k,l}}(\tau)$ denote the inverse image of $Y(w_{1,l})$ (resp.\ $Y(w_{1,n-k+2})$) under the iterated fibration in Theorem \ref{non-empty}.
Then $Y(w_{k,l})$ is an irreducible component of $X_{w_{k,l}}(\tau)$ (cf.\ Remark \ref{spherical rema}).

\begin{theo}
\label{main theo}
Let $w_{k,l}\in \SAdm(\mu)_{\DL}$.
Then
$$\pi(X_{w_{k,l}}(b))\cong\pi_{S(w_{k,l},\sigma)}(X_{w_{k,l}}(b))=\bigsqcup_{j\in G(F)/G(F)\cap P_{\supp_\sigma(w_{k,l})_1\cup S(w_{k,l},\sigma)_1}} j\tau_1\pi_{S(w_{k,l},\sigma)}(Y(w_{k,l})).$$
Here $\supp_\sigma(w_{k,l})$ and $S(w_{k,l},\sigma)$ are as in Lemma \ref{SwDL}.

Let $w_{k,l}\in \SAdm(\mu)_{\neq\DL}$.
Then $\pi(X_{w_{k,l}}(b))$ is $G(F)$-equivariant universally homeomorphic to a Zariski-locally trivial $\A^1$-bundle over $\pi(X_{w'_{k,l}}(b))$.
In particular, if $k+l\le n+2$ (resp.\ $k+l\geq n+3$), then 
$$\pi(X_{w_{k,l}}(b))\cong\pi_{S(w_{k,l},\sigma)}(X_{w_{k,l}}(b))=\bigsqcup_{j\in G(F)/G(\cO_F)} j\tau_1\pi_{S(w_{k,l},\sigma)}(Y(w_{k,l}))$$
and $\tau_1\pi_{S(w_{k,l},\sigma)}(Y(w_{k,l}))$ is an irreducible component, which is $G(\cO_F)$-equivariant universally homeomorphic to an iterated fibration of rank $\frac{k-1}{2}$ (resp.\ $k+\frac{l-n-3}{2}$) over $\tau_1\pi_{S(w_{k,l},\sigma)}(Y(w_{1,l}))$ (resp.\ $\tau_1\pi_{S(w_{k,l},\sigma)}(Y(w_{1,n-k+2}))$).
\end{theo}
\begin{proof}
Let $w_{k,l}\in \SAdm(\mu)_{\DL}$.
By Proposition \ref{spherical}, we have
$$\pi_{S(w_{k,l},\sigma)}(X_{w_{k,l}}(\tau))=\bigsqcup_{j\in \J_\tau/\J_\tau\cap P_{\supp_\sigma(w_{k,l})\cup S(w_{k,l},\sigma)}} j\pi_{S(w_{k,l},\sigma)}(Y(w_{k,l})).$$
By $\tau_1^{-1}b\sigma(\tau_1)=\tau$, we have isomorphisms $X_\mu(\tau)\xrightarrow{\sim}X_\mu(b), xK\mapsto \tau_1xK$ and $\J_{\tau}\xrightarrow{\sim}\J_b=G(F), j\mapsto \tau_1 j\tau_1^{-1}$.
Under these isomorphisms, $\pi_{S(w_{k,l},\sigma)}(X_{w_{k,l}}(\tau))$ maps to
$$\pi_{S(w_{k,l},\sigma)}(X_{w_{k,l}}(b))=\bigsqcup_{j\in G(F)/G(F)\cap P_{\supp_\sigma(w_{k,l})_1\cup S(w_{k,l},\sigma)_1}} j\tau_1\pi_{S(w_{k,l},\sigma)}(Y(w_{k,l})).$$
The first statement follows from this and $\pi_{S(w_{k,l},\sigma)}(X_{w_{k,l}}(b))\cong \pi(X_{w_{k,l}}(b))$.

Let $w_{k,l}\in \SAdm(\mu)_{\neq\DL}$.
By Lemma \ref{SwnDL} and Proposition \ref{DL method prop parahoric}, $\pi_{S(w_{k,l},\sigma)}(X_{w_{k,l}}(b))$ is $G(F)$-equivariant universally homeomorphic to a Zariski-locally trivial $\A^1$-bundle over $\pi_{S(w_{k,l},\sigma)}(X_{w'_{k,l}}(b))$.
In particular, if $w_{k,l}\in\SAdm(\mu)_{\neq\DL}$ satisfies (i) (resp.\ (ii)) of Theorem \ref{non-empty}, then $\pi_{S(w_{k,l},\sigma)}(X_{w_{k,l}}(b))$ is an iterated fibration of rank $\frac{k-1}{2}$ (resp.\ $k+\frac{l-n-3}{2}$) over $\pi_{S(w_{k,l},\sigma)}(X_{w_{1,l}}(b))$ (resp.\ $\pi_{S(w_{k,l},\sigma)}(X_{w_{1,n-k+2}}(b))$).
By Lemma \ref{SwDL}, we have $\supp_\sigma(w_{k,l})_1\cup S(w_{k,l},\sigma)_1=S$ if $k=1$ and $l\geq \frac{n+3}{2}$.
Thus the last statement follows from these results and the isomorphism $\pi_{S(w_{k,l},\sigma)}(X_{w_{k,l}}(b))\cong \pi(X_{w_{k,l}}(b))$.
This finishes the proof.
\end{proof}

\begin{rema}
If $w_{k,l}\in \SAdm(\mu)_{\DL}$, then we set $P_{k,l}=P_{\supp_\sigma(w_{k,l})_1\cup S(w_{k,l},\sigma)_1}$ and $X_{k,l}=Y_{k,l}=\tau_1\pi(Y(w_{k,l}))\cong \tau_1\pi_{S(w_{k,l},\sigma)}(Y(w_{k,l}))$.
If $w_{k,l}\in \SAdm(\mu)_{\neq\DL}$ and $k+l\le n+2$ (resp.\ $k+l\geq n+3$), then we set $P_{k,l}=G(\cO)$, $X_{k,l}=\tau_1\pi(Y(w_{1,l}))\cong \tau_1\pi_{S(w_{k,l},\sigma)}(Y(w_{1,l}))$ (resp.\ $X_{k,l}=\tau_1\pi(Y(w_{1,n-k+2}))\cong\tau_1\pi_{S(w_{k,l},\sigma)}(Y(w_{1,n-k+2}))$) and $Y_{k,l}=\tau_1\pi(Y(w_{k,l}))\cong \tau_1\pi_{S(w_{k,l},\sigma)}(Y(w_{k,l}))$.
This notation justifies Theorem \ref{stratification}.
\end{rema}

\begin{rema}
\label{DL U}
Let $w_{k,l}\in \SAdm(\mu)_{\neq \emptyset}$.
There is an isomorphism $G(L)/P_{S(w_{k,l},\sigma)}\xrightarrow{\sim} G(L)/P_{S(w_{k,l},\sigma)_1}$
given by $gP_{S(w_{k,l},\sigma)}\mapsto gP_{S(w_{k,l},\sigma)}\tau_1^{-1}=g\tau_1^{-1}P_{S(w_{k,l},\sigma)_1}$.
Under this isomorphism, $\tau_1P_{\supp_\sigma(w_{k,l})\cup S(w_{k,l},\sigma)}/P_{S(w_{k,l},\sigma)}$ maps to $P_{\supp_\sigma(w_{k,l})_1\cup S(w_{k,l},\sigma)_1}/P_{S(w_{k,l},\sigma)_1}$.
So if $w_{k,l}\in \SAdm(\mu)_{\DL}$, then $\tau_1\pi_{S(w_{k,l},\sigma)}(Y(w_{k,l}))$ maps to
$$\pi_{S(w_{k,l},\sigma)_1}(Y(w_{k,l}^0))=\{gP_{S(w_{k,l},\sigma)_1}\mid g^{-1}\sigma(g)\in P_{S(w_{k,l},\sigma)_1}\cdot_\sigma Iw_{k,l}^0I\},$$
where $w_{k,l}^0=b^{-1}\tau_1 w_{k,l}\sigma(\tau_1)^{-1}\in W_a$.
Note that this is a $G(F)\cap P_{\supp_\sigma(w_{k,l})_1\cup S(w_{k,l},\sigma)_1}$-equivariant isomorphism.
If, moreover, $\supp_\sigma(w_{k,l})_1\cup S(w_{k,l},\sigma)_1=S$, then we have $w_{k,l}^0\in W_0$ and $\pi_{S(w_{k,l},\sigma)_1}(Y(w_{k,l}^0))$ can be identified with a Deligne-Lusztig variety associated to $w_{k,l}^0$ in the partial flag variety of type $S(w_{k,l},\sigma)_1$ for $G(\cO/\vp)=G(\aFq)$.
Under this identification, the action of $G(\cO_F)$ factors through $G(\cO_F/\vp)=G(\F_q)$, which coincides with the usual action on Deligne-Lusztig varieties for $G(\aFq)$.
Thus if $w_{k,l}\in \SAdm(\mu)_{\neq \DL}$, then $\tau_1\pi_{S(w_{k,l},\sigma)}(Y(w_{k,l}))$ is $G(\cO_F)$-equivariant universally homeomorphic to an iterated fibration over such a Deligne-Lusztig variety.
\end{rema}

\begin{coro}
\label{Coxeter coro}
Assume that $n$ is odd (resp.\ even) and that $w_{k,l}\in \SAdm(\mu)_{\neq \DL}$.
If $k=\frac{n+1}{2}$ or $l=\frac{n+3}{2}$ (resp.\ $k=\frac{n}{2}$ or $l=\frac{n}{2}+2$), then $S(w_{k,l},\sigma)=\emptyset$.
Moreover, the iterated fibration corresponding to $w_{k,l}$ in Theorem \ref{main theo} decomposes into the product of a Deligne-Lusztig variety of Coxeter type and a finite-dimensional affine space.
\end{coro}
\begin{proof}
This is a combination of Remark \ref{positive Coxeter}, Lemma \ref{SwnDL} and \cite[Theorem 5.20]{SSY23}.
\end{proof}


We write $w_{k,l}\geq_{S,\sigma}w_{k',l'}$ if there exists $u\in W_0$ such that $w_{k,l}\geq u^{-1}w_{k',l'}\sigma(u)$.
This is a partial order on $\SW$.
See \cite[Proposition 3.13]{He07a} and \cite[Corollary 4.6]{He07b}.
\begin{lemm}
\label{closure relation lemm}
Let $w_{k,l},w_{k',l'}\in \SAdm(\mu)_{\DL}$ with $w_{k,l}\geq_{S,\sigma}w_{k',l'}$.
If $k,k'\geq 2$, then $k\geq k'$ and $l\geq l'$.
\end{lemm}
\begin{proof}
By \cite[Proposition 3.4]{ABFGGN24}, $w_{k,l}\geq w_{k',l'}$ is equivalent to saying that $k\geq k'$ and $l\geq l'$.
So it is enough to show that $w_{k,l}\geq_{S,\sigma}w_{k',l'}$ implies $w_{k,l}\geq w_{k',l'}$.

Assume that $k,k'\geq 2$.
Since the case $l'\le k$ is obvious, we may also assume that $l'>k$.
By $w_{k',l'}\in \SW$ and \cite[Proposition 3.8]{He07a}, there exists $u\in W_0$ such that $\ell(w_{k',l'}\sigma(u))=\ell(w_{k',l'})-\ell(u)$ and $w_{k,l}\geq u^{-1}w_{k',l'}\sigma(u)$.
Let $s_i\in \supp(u)$.
It follows from $k'\geq 2$ and $\ell(w_{k',l'}\sigma(u))=\ell(w_{k',l'})-\ell(u)$ that $n-i-2\le l'-3$.
It follows from $k\geq 2$ and $u^{-1}w_{k',l'}\sigma(u)\le w_{k,l}$ that $i\le l-3$ or $i=n-1$.
Thus $n-l'+1\le i\le l-3$ or $i=n-1$.
By Theorem \ref{non-empty} and $k,k'\geq 2$, we have $l+l'\le n+2$.
This implies $i=n-1$.
On the other hand, we can easily check that $u\neq s_{n-1}$ by $l'>k$ (and \cite[Theorem 2.1.5]{BB05} similarly as the proof of \cite[Proposition 3.4]{ABFGGN24}).
Thus $u=1$ and hence $w_{k,l}\geq w_{k',l'}$.
This finishes the proof.
\end{proof}


As explained in \cite[\S2.4]{GHN24}, the closure relations in the fully Hodge-Newton decomposable cases are described by (i) the partial order $\geq_{S,\sigma}$ and (ii) the intersection of indices.
In our case, the same description is possible for $w_{k,l}\in \SAdm(\mu)_{\DL}$, and (i) can be expressed more explicitly.
\begin{coro}
\label{closure relation}
Let $w_{k,l}\in \SAdm(\mu)_{\DL}$.
Then $j\tau_1\overline{\pi(Y(w_{k,l}))}$ is a union of strata in Theorem \ref{main theo}.
For $w_{k',l'}\in \SAdm(\mu)_{\neq \emptyset}$, $j\tau_1\overline{\pi(Y(w_{k,l}))}$ contains a stratum $j'\tau_1\pi(Y(w_{k',l'}))$ if and only if the following two conditions are both satisfied:
\begin{enumerate}[(i)]
\item $k\geq k'$ and $l\geq l'$.
\item $j(G(F)\cap P_{\supp_\sigma(w_{k,l})_1\cup S(w_{k,l},\sigma)_1})\cap j'(G(F)\cap P_{\supp_\sigma(w_{k',l'})_1\cup S(w_{k',l'},\sigma)_1})\neq \emptyset$.
\end{enumerate}
In particular, $w_{k',l'}\in \SAdm(\mu)_{\DL}$ if this is the case.
\end{coro}
\begin{proof}
We first prove that (i) is equivalent to $\text{(i)}'$ $w_{k,l}\geq_{S,\sigma}w_{k',l'}$.
We clearly have $\text{(i)}\Rightarrow \text{(i)}'$.
It follows from $w_{k',l'}\in \SW$ and \cite[Proposition 3.8]{He07a} that if $w_{k,l}\geq_{S,\sigma}w_{k',l'}$, then $\supp_\sigma(w_{k,l})\supseteq \supp_\sigma(w_{k',l'})$.
Thus $\text{(i)}'$ combined with the explicit description of $\supp_\sigma(w_{k,l})$ implies that $l\geq l'$, and that if $k=1$, then $k'=1$.
Then the equivalence of (i) and $\text{(i)}'$ follows from Lemma \ref{closure relation lemm}.

Let $\tilde j,\tilde j'\in \J_\tau$.
Similarly as the proof of \cite[Theorem 7.2.1]{GH15} (see also \cite[Theorem 3.1(2)]{He09}), we can prove that $\tilde j\overline{\pi(Y(w_{k,l}))}$ is a union of strata in $X_\mu(\tau)$, and that it contains a stratum $\tilde j'\pi(Y(w_{k',l'}))$ if and only if $\text{(i)}'$ $w_{k,l}\geq_{S,\sigma}w_{k',l'}$ and $\text{(ii)}'$ $\tilde j(\J_\tau\cap P_{\supp_\sigma(w_{k,l})\cup S(w_{k,l},\sigma)})\cap \tilde j'(\J_\tau\cap P_{\supp_\sigma(w_{k',l'})\cup S(w_{k',l'},\sigma)})\neq \emptyset$.
Clearly, $\text{(ii)}'$ is equivalent to (ii), which finishes the proof.
\end{proof}

\begin{rema}
The equivalence of $\text{(i)}$ and $\text{(i)}'$ in the proof of Corollary \ref{closure relation} proves a part of \cite[Conjecture 3.23]{ABFGGN24}.
\end{rema}

\subsection{Irreducible Components}
In this subsection, we compare Theorem \ref{main theo} and the results in \cite{FI21} and \cite{FHI23}.
Let us denote by $\Irr X_\mu(b)$ the set of irreducible components of $X_\mu(b)$.
Note that $X_\mu(b)$ is known to be equidimensional (cf.\ \cite{HV12}, \cite{Takaya22}).
The following result is consistent with \cite[Proposition 8.1]{FI21}.
\begin{theo}
\label{irreducible components theo}
Let $n\ge 3$.
We have $\dim X_\mu(b)=n-2$ and $|G(F)\backslash \Irr X_\mu(b)|=\lfloor\frac{n}{2}\rfloor$.
If $n$ is odd, then
$$\Irr X_\mu(b)=\bigsqcup_{j\in G(F)/G(\cO_F)}(\{j\tau_1\overline{\pi(Y(w_{1,n}))}\}\sqcup \{j\tau_1\overline{\pi(Y(w_{k,n-1}))}\mid 3\le k\le \tfrac{n+1}{2}\}).$$
If $n$ is even, then
\begin{align*}
\Irr X_\mu(b)=&\bigsqcup_{j\in G(F)/G(\cO_F)}(\{j\tau_1\overline{\pi(Y(w_{1,n}))}\}\sqcup \{j\tau_1\overline{\pi(Y(w_{k,n-1}))}\mid 3\le k\le \tfrac{n+1}{2}\})\\
&\sqcup\bigsqcup_{j\in G(F)/G(F)\cap P}\{j\tau_1\overline{\pi(Y(w_{\frac{n}{2},\frac{n}{2}+1}))}\},
\end{align*}
where $P=P_{\tS\setminus\{s_{\frac{n}{2}}\}}$.
\end{theo}
\begin{proof}
By Theorem \ref{main theo}, we have $\dim X_\mu(b)=\dim\pi(Y(w_{1,n}))=\dim \pi(Y(w_{k,n-1}))=n-2$ for $3\le k\le \frac{n+1}{2}$.
We also have $\dim \pi(Y(w_{\frac{n}{2},\frac{n}{2}+1}))=n-2$ if $n$ is even.
Again by Theorem \ref{main theo} and Remark \ref{spherical rema}, $\pi(Y(w_{1,n}))$ and $ \pi(Y(w_{k,n-1}))$ for $3\le k\le \frac{n+1}{2}$ are irreducible.
The same is true for $\pi(Y(w_{\frac{n}{2},\frac{n}{2}+1}))$ if $n$ is even.
So the $G(F)$-orbits of the closure of these varieties are irreducible components of $X_\mu(b)$.
Since the decomposition of $X_\mu(b)$ into locally closed subvarieties by Theorem \ref{main theo} is a disjoint union, we obtain the theorem.
\end{proof}

Unlike the case $w_{k,l}\in \SAdm(\mu)_{\DL}$ in Corollary \ref{closure relation}, the closure of $j\tau_1\pi(Y(w_{k,l}))$ is hard to describe explicitly if $w_{k,l}\in \SAdm(\mu)_{\neq\DL}$.
This is already true for the case $n=5$, which we discuss in the following example.
\begin{exam}
\label{counterexample}
Assume that $n=5$.
We consider two irreducible components $\tau_1\overline{\pi(Y(w_{3,4}))}$ and $s_0\tau_1\overline{\pi(Y(w_{1,5}))}$ lying in different $G(F)$-orbits.
We have $w_{1,5}=s_0s_1s_2\tau\xrightarrow{s_3s_0s_1}s_1\tau$.
By \cite[Corollary 2.5]{Lusztig76} (see also \cite[Proposition 1.1]{Gortz19}), $Y(s_1\tau)\subset Is_1s_2s_1I/I$.
Since $s_1s_2s_1s_3s_0s_1$ is a reduced expression, we have $$s_0\tau_1(X_{w_{1,5}}(\tau) \cap Is_1s_2s_1s_3s_0s_1I/I)=s_0\tau_1 X_{w_{1,5}}(\tau) \cap Is_0s_2s_3s_2s_4s_1s_2\tau_1I/I\neq \emptyset$$ by Proposition \ref{DL method prop}.
In particular,
$$s_0\tau_1\pi(X_{w_{1,5}}(\tau))\cap I\vp^{((1,0,0,1,-1),0)}K/K\neq \emptyset.$$
On the other hand, the irreducible components of $X_\mu(b)$ corresponding to $G(\cO_F)$ are the closure of $X_\mu(b)\cap K\vp^{((1,0,0,0,0),0)}K/K$ and $X_\mu(b)\cap K\vp^{((1,1,0,0,-1),0)}K/K$ by the proof of \cite[Theorem 6.2.2]{FHI23} (note that their $\mu$ is the dual of our $\mu$).
Since $\tau_1\pi(Y(w_{1,5}))\subset K\vp^{((1,0,0,0,0),0)}K/K$, $\tau_1\overline{\pi(Y(w_{3,4}))}$ must be equal to the closure of $X_\mu(b)\cap K\vp^{((1,1,0,0,-1),0)}K/K$.
Thus $\tau_1\overline{\pi(Y(w_{3,4}))}\cap s_0\tau_1\pi(Y(w_{1,5}))\neq \emptyset$.
Note that $s_0\tau_1\pi(Y(w_{1,5}))\nsubseteq \tau_1\overline{\pi(Y(w_{3,4}))}$.
Therefore neither $\tau_1\overline{\pi(Y(w_{3,4}))}$ nor $\tau_1\overline{\pi(Y(w_{3,4}))}\cap s_0\tau_1\overline{\pi(Y(w_{1,5}))}$ can be written as a union of strata in Theorem \ref{main theo}.
\end{exam}

It would be interesting to study in more detail the relationship between the description of irreducible components of $X_\mu(b)$ in this paper and the descriptions in \cite{FI21} and \cite{FHI23} by Fox, Howard and Imai.
In \cite{FI21} and \cite{FHI23}, the authors exploit the explicit construction of irreducible components given by Xiao-Zhu \cite{XZ17}.
This construction seems to have a connection with the $\J$-stratification of affine Deligne-Lusztig varieties introduced in \cite{CV18} by Chen-Viehmann (see also the comment at the end of \cite[\S1]{CV18} by the authors).
On the other hand, the stratification obtained in this paper is a refinement of the Ekedahl-Oort stratification of $X_\mu(b)$.
From this point of view, the comparison between this paper and the work by Fox-Howard-Imai should be a special case of the comparison of the Ekedahl-Oort stratification and the $\J$-stratification, both of which exist in general.
See \cite{Shimada4} for the comparison of these two stratifications in the superbasic case.

\section{An analogue of Oort's conjecture}
\label{Oort analogue section}
From this section until the end of this paper, we assume that $F=\Q_p$, $p\geq3$ and $n\geq3$.
Then the center of $G(\Q_p)$ is $\Q_{p^2}^{\times}$.
Let $\breve{\Q}_p$ be the completion of the maximal unramified extension of $\Q_p$, and let $\Z_{p^2}$ and $\breve{\Z}_p$ be the rings of integers of $\Q_{p^2}$ and $\breve{\Q}_p$, respectively.
The goal of this section is to prove the following theorem.

\begin{theo}
\label{Oort analogue}
There exists an open dense subscheme $U\subseteq X_\mu(b)$ such that, for every closed point $xK\in U$, if $j\in G(\Q_p)$ has finite order and $jxK=xK$, then $j$ is a unit-root scalar in $\Z_{p^2}^{\times}$.
\end{theo}
Note that the unit roots in $\Z_{p^2}^{\times}$ are the Teichm\"uller lifts of the elements of $\F_{p^2}^{\times}$.

\begin{lemm}
\label{torsion free hyperspecial}
Let $G(\Z_p)^+$ be the kernel of the reduction map $G(\Z_p)\twoheadrightarrow G(\F_p)$.
Then $G(\Z_p)^+$ is torsion-free.
\end{lemm}
\begin{proof}
We have $G(\Z_p)\subset\GL_n(\Z_{p^2})$.
Choosing a $\Z_p$-basis of $\Z_{p^2}$, we regard $\GL_n(\Z_{p^2})$ as a subgroup of $\GL_{2n}(\Z_p)$.
Under this inclusion, $G(\Z_p)^+$ is contained in $1+p\operatorname{Mat}_{2n}(\Z_p)$.
Since $p\geq3$, the latter group is torsion-free by \cite[Theorem 4.5 \& Theorem 5.2]{DDMS99}.
This finishes the proof.
\end{proof}

\begin{proof}[Proof of Theorem \ref{Oort analogue}]
By Theorem \ref{irreducible components theo}, it suffices to find, in each top-dimensional Ekedahl--Oort stratum, an open dense subscheme such that every finite-order element of $G(\Q_p)$ fixing one of its closed points is a unit-root scalar.
It is enough to construct such subschemes in
$$\tau_1\pi(Y(w_{1,n}))\quad\text{and}\quad \tau_1\pi(Y(w_{k,n-1}))\quad \left(3\leq k\leq\frac{n+1}{2}\right),$$
and, when $n$ is even, in
$$b\tau_1^{1-\frac{n}{2}}\pi\left(Y\left(w_{\frac{n}{2},\frac{n}{2}+1}\right)\right).$$
Note that the stabilizer in $G(\Q_p)$ of each of these subvarieties is $G(\Z_p)$.
For each $j\in G(\Z_p)$, the fixed-point locus on each of the above subvarieties is closed.
By \cite[Lemma 2.1]{STOort}, the action of $G(\Z_p)$ on each of them factors through a finite quotient.
Thus it remains to show that if $j\in G(\Z_p)$ has finite order and fixes one of them pointwise, then $j$ is a unit-root scalar.

Let $j$ be such an element.
Then, by Theorem \ref{main theo} and Remark \ref{DL U}, the reduction $\overline j\in G(\F_p)$ fixes a Deligne--Lusztig variety pointwise.
By the proof of \cite[Proposition 2.3]{STOort}, the element $\overline j$ lies in the center $\F_{p^2}^{\times}$ of $G(\F_p)$.
Let $z\in\Z_{p^2}^{\times}$ be its Teichm\"uller lift.
Then $z^{-1}j\in G(\Z_p)^+$ has finite order and hence is trivial by Lemma \ref{torsion free hyperspecial}.
Therefore $j=z$.
This finishes the proof.
\end{proof}

\bibliographystyle{myamsplain}
\bibliography{reference}

\providecommand{\bysame}{\leavevmode\hbox to3em{\hrulefill}\thinspace}
\providecommand{\MR}{\relax\ifhmode\unskip\space\fi MR }
\providecommand{\MRhref}[2]{%
  \href{http://www.ams.org/mathscinet-getitem?mr=#1}{#2}
}
\providecommand{\href}[2]{#2}
\begin{thebibliography}{10}

\bibitem{ABFGGN24}
E.~Andrews, D.~Bhamidipati, M.~Fox, H.~Goodson, S.~R. Groen, and S.~Nair,
  \emph{{E}kedahl-{O}ort strata in the $\mathsf{GU}(q-2,2)$ {S}himura variety},
  arXiv:2405.04464 (2024).

\bibitem{ABFGGN25}
\bysame, \emph{The {E}kedahl-{O}ort and {N}ewton stratification of the
  $\mathsf{GU}(3,2)$ {S}himura variety}, arXiv:2510.01090 (2025).

\bibitem{BS17}
B.~Bhatt and P.~Scholze, \emph{Projectivity of the {W}itt vector affine
  {G}rassmannian}, Invent. Math. \textbf{209} (2017), no.~2, 329--423.
  \MR{3674218}

\bibitem{BB05}
A.~Bj\"orner and F.~Brenti, \emph{Combinatorics of {C}oxeter groups}, Graduate
  Texts in Mathematics, vol. 231, Springer, New York, 2005. \MR{2133266}

\bibitem{BR06}
C.~Bonnaf\'e and R.~Rouquier, \emph{On the irreducibility of
  {D}eligne-{L}usztig varieties}, C. R. Math. Acad. Sci. Paris \textbf{343}
  (2006), no.~1, 37--39. \MR{2241956}

\bibitem{BT72}
A.~Borel and J.~Tits, \emph{Compl\'ements \`a{} l'article: ``{G}roupes
  r\'eductifs''}, Inst. Hautes \'Etudes Sci. Publ. Math. (1972), no.~41,
  253--276. \MR{315007}

\bibitem{CV18}
M.~Chen and E.~Viehmann, \emph{Affine {D}eligne-{L}usztig varieties and the
  action of {$J$}}, J. Algebraic Geom. \textbf{27} (2018), no.~2, 273--304.
  \MR{3764277}

\bibitem{DL76}
P.~Deligne and G.~Lusztig, \emph{Representations of reductive groups over
  finite fields}, Ann. of Math. (2) \textbf{103} (1976), no.~1, 103--161.
  \MR{393266}

\bibitem{DDMS99}
J.~D. Dixon, M.~P.~F. du~Sautoy, A.~Mann, and D.~Segal, \emph{Analytic
  pro-{$p$} groups}, second ed., Cambridge Studies in Advanced Mathematics,
  vol.~61, Cambridge University Press, Cambridge, 1999.

\bibitem{FHI23}
M.~Fox, B.~Howard, and N.~Imai, \emph{{R}apoport--{Z}ink spaces of type {GU}(2,
  n - 2)}, Algebraic Geometry \textbf{13} (2026), no.~2, 193--240.

\bibitem{FI21}
M.~Fox and N.~Imai, \emph{The supersingular locus of the {S}himura variety of
  $\mathrm{GU}(2,n-2)$}, arXiv:2108.03584, to appear in Peking Math. J.

\bibitem{Gortz19}
U.~G\"{o}rtz, \emph{Stratifications of affine {D}eligne-{L}usztig varieties},
  Trans. Amer. Math. Soc. \textbf{372} (2019), no.~7, 4675--4699. \MR{4009395}

\bibitem{GH10}
U.~G\"{o}rtz and X.~He, \emph{Dimensions of affine {D}eligne-{L}usztig
  varieties in affine flag varieties}, Doc. Math. \textbf{15} (2010),
  1009--1028. \MR{2745691}

\bibitem{GH15}
\bysame, \emph{Basic loci of {C}oxeter type in {S}himura varieties}, Camb. J.
  Math. \textbf{3} (2015), no.~3, 323--353. \MR{3393024}

\bibitem{GHN15}
U.~G\"{o}rtz, X.~He, and S.~Nie, \emph{{$\bf P$}-alcoves and nonemptiness of
  affine {D}eligne-{L}usztig varieties}, Ann. Sci. \'{E}c. Norm. Sup\'{e}r. (4)
  \textbf{48} (2015), no.~3, 647--665. \MR{3377055}

\bibitem{GHN19}
\bysame, \emph{Fully {H}odge-{N}ewton decomposable {S}himura varieties}, Peking
  Math. J. \textbf{2} (2019), no.~2, 99--154. \MR{4060001}

\bibitem{GHN24}
U.~G\"ortz, X.~He, and S.~Nie, \emph{Basic loci of {C}oxeter type with
  arbitrary parahoric level}, Canad. J. Math. \textbf{76} (2024), no.~1,
  126--172. \MR{4687768}

\bibitem{GHR20}
U.~G\"{o}rtz, X.~He, and M.~Rapoport, \emph{Extremal cases of {R}apoport-{Z}ink
  spaces}, Journal of the Institute of Mathematics of Jussieu (2020), 1--56.

\bibitem{HV18}
P.~Hamacher and E.~Viehmann, \emph{Irreducible components of minuscule affine
  {D}eligne-{L}usztig varieties}, Algebra Number Theory \textbf{12} (2018),
  no.~7, 1611--1634. \MR{3871504}

\bibitem{HV12}
U.~Hartl and E.~Viehmann, \emph{Foliations in deformation spaces of local
  {$G$}-shtukas}, Adv. Math. \textbf{229} (2012), no.~1, 54--78. \MR{2854170}

\bibitem{HR17}
X.~He and M.~Rapoport, \emph{Stratifications in the reduction of {S}himura
  varieties}, Manuscripta Math. \textbf{152} (2017), no.~3-4, 317--343.
  \MR{3608295}

\bibitem{He07a}
X.~He, \emph{The {$G$}-stable pieces of the wonderful compactification}, Trans.
  Amer. Math. Soc. \textbf{359} (2007), no.~7, 3005--3024. \MR{2299444}

\bibitem{He07b}
\bysame, \emph{Minimal length elements in some double cosets of {C}oxeter
  groups}, Adv. Math. \textbf{215} (2007), no.~2, 469--503. \MR{2355597}

\bibitem{He09}
\bysame, \emph{{$G$}-stable pieces and partial flag varieties}, Representation
  theory, Contemp. Math., vol. 478, Amer. Math. Soc., Providence, RI, 2009,
  pp.~61--70. \MR{2513266}

\bibitem{HTX17}
D.~Helm, Y.~Tian, and L.~Xiao, \emph{Tate cycles on some unitary {S}himura
  varieties mod $p$}, Algebra Number Theory \textbf{11} (2017), no.~10,
  2213--2288. \MR{3744356}

\bibitem{HP14}
B.~Howard and G.~Pappas, \emph{On the supersingular locus of the {${\rm
  GU}(2,2)$} {S}himura variety}, Algebra Number Theory \textbf{8} (2014),
  no.~7, 1659--1699. \MR{3272278}

\bibitem{Kottwitz85}
R.~E. Kottwitz, \emph{Isocrystals with additional structure}, Compositio Math.
  \textbf{56} (1985), no.~2, 201--220. \MR{809866}

\bibitem{KR11}
S.~Kudla and M.~Rapoport, \emph{Special cycles on unitary {S}himura varieties
  {I}. {U}nramified local theory}, Invent. Math. \textbf{184} (2011), no.~3,
  629--682. \MR{2800697}

\bibitem{Lansky01}
J.~M. Lansky, \emph{Decomposition of double cosets in {$\mathfrak p$}-adic
  groups}, Pacific J. Math. \textbf{197} (2001), no.~1, 97--117. \MR{1810210}

\bibitem{Lim23}
D.~G. Lim, \emph{Nonemptiness of single affine {D}eligne-{L}usztig varieties},
  arXiv:2302.04976 (2023).

\bibitem{Lusztig76}
G.~Lusztig, \emph{{C}oxeter orbits and eigenspaces of {F}robenius}, Invent.
  Math. \textbf{38} (1976/77), no.~2, 101--159. \MR{453885}

\bibitem{PR08}
G.~Pappas and M.~Rapoport, \emph{Twisted loop groups and their affine flag
  varieties}, Adv. Math. \textbf{219} (2008), no.~1, 118--198, With an appendix
  by T. Haines and Rapoport. \MR{2435422}

\bibitem{PSV26}
P.~Philippe, F.~Schnelle, and E.~Viehmann, \emph{{O}ort's conjecture for split
  unitary {S}himura varieties}, arXiv:2609.04072 (2026).

\bibitem{Schremmer23}
F.~Schremmer, \emph{Newton strata in {L}evi subgroups}, Manuscripta Math.
  \textbf{175} (2024), no.~1-2, 513--519. \MR{4790570}

\bibitem{SSY23}
F.~Schremmer, R.~Shimada, and Q.~Yu, \emph{Affine {D}eligne-{L}usztig varieties
  of positive {C}oxeter type}, arXiv:2312.02630 (2026).

\bibitem{Shimada4}
R.~Shimada, \emph{The {E}kedahl-{O}ort stratification and the semi-module
  stratification}, arXiv:2309.03371, to appear in Canad. J. Math.

\bibitem{Shimada5}
\bysame, \emph{Basic loci of positive {C}oxeter type for $\mathrm{GL}_n$},
  arXiv:2402.13216 (2024).

\bibitem{ST24}
R.~Shimada and T.~Takamatsu, \emph{On the supersingular locus of the {S}iegel
  modular variety of genus $3$ or $4$}, arXiv:2403.19505, to appear in Ann.
  Inst. Fourier (Grenoble).

\bibitem{STOort}
\bysame, \emph{Some generalizations of {O}ort's conjecture}, arXiv:2609.03188
  (2026).

\bibitem{Takaya22}
Y.~Takaya, \emph{Equidimensionality of affine {D}eligne-{L}usztig varieties in
  mixed characteristic}, Adv. Math. \textbf{465} (2025), Paper No. 110153, 22.
  \MR{4864699}

\bibitem{TY22}
Y.~Terakado and C.-F. Yu, \emph{Mass formulas and the basic locus of unitary
  {S}himura varieties}, arXiv:2210.04054 (2022).

\bibitem{TX19}
Y.~Tian and L.~Xiao, \emph{Tate cycles on some quaternionic {S}himura varieties
  {${\rm mod}\,p$}}, Duke Math. J. \textbf{168} (2019), no.~9, 1551--1639.
  \MR{3961211}

\bibitem{Trentin23}
S.~Trentin, \emph{On the {R}apoport-{Z}ink space for {$\rm GU(2,4)$} over a
  ramified prime}, J. Reine Angew. Math. \textbf{826} (2025), 143--223.
  \MR{4959572}

\bibitem{Viehmann26}
E.~Viehmann, \emph{Oort's conjecture on automorphisms of generic supersingular
  abelian varieties}, arXiv:2603.06033 (2026).

\bibitem{VW11}
I.~Vollaard and T.~Wedhorn, \emph{The supersingular locus of the {S}himura
  variety of {${\rm GU}(1,n-1)$} {II}}, Invent. Math. \textbf{184} (2011),
  no.~3, 591--627. \MR{2800696}

\bibitem{Wang21}
H.~Wang, \emph{Deligne-{L}usztig varieties and basic {EKOR} strata}, Canad.
  Math. Bull. \textbf{64} (2021), no.~2, 349--367. \MR{4273205}

\bibitem{XZ17}
L.~Xiao and X.~Zhu, \emph{Cycles on {S}himura varieties via geometric
  {S}atake}, arXiv:1707.05700 (2017).

\bibitem{Zhang12}
W.~Zhang, \emph{On arithmetic fundamental lemmas}, Invent. Math. \textbf{188}
  (2012), no.~1, 197--252. \MR{2897697}

\bibitem{Zhu17}
X.~Zhu, \emph{Affine {G}rassmannians and the geometric {S}atake in mixed
  characteristic}, Ann. of Math. (2) \textbf{185} (2017), no.~2, 403--492.
  \MR{3612002}

\end{thebibliography}
\end{document}